\documentclass[12pt,oneside]{amsart}  
\usepackage{amssymb}
\usepackage{amsfonts}
\usepackage{amsthm}
\usepackage{amsthm}
\usepackage{amscd}

\reversemarginpar         
\numberwithin{equation}{section}
\theoremstyle{plain}
\newtheorem{theorem}{Theorem}[section]

\newtheorem{corollary}[theorem]{Corollary}
\newtheorem{lemma}[theorem]{Lemma}
\theoremstyle{definition}

\theoremstyle{remark}
\newtheorem*{remark}{Remark}

\begin{document}
%

\newcommand{\M}{\mathcal{M}_{g,N+1}^{(1)}}
\newcommand{\Teich}{\mathcal{T}_{g,N+1}^{(1)}}
\newcommand{\T}{\mathrm{T}}
\newcommand{\corr}{\bf}
\newcommand{\vac}{|0\rangle}
\newcommand{\Ga}{\Gamma}
\newcommand{\new}{\bf}
\newcommand{\define}{\def}
\newcommand{\redefine}{\def}
\newcommand{\Cal}[1]{\mathcal{#1}}
\renewcommand{\frak}[1]{\mathfrak{{#1}}}
\newcommand{\refE}[1]{(\ref{E:#1})}
\newcommand{\refS}[1]{Section~\ref{S:#1}}
\newcommand{\refSS}[1]{Section~\ref{SS:#1}}
\newcommand{\refT}[1]{Theorem~\ref{T:#1}}
\newcommand{\refO}[1]{Observation~\ref{O:#1}}
\newcommand{\refP}[1]{Proposition~\ref{P:#1}}
\newcommand{\refD}[1]{Definition~\ref{D:#1}}
\newcommand{\refC}[1]{Corollary~\ref{C:#1}}
\newcommand{\refL}[1]{Lemma~\ref{L:#1}}
\newcommand{\R}{\ensuremath{\mathbb{R}}}
\newcommand{\C}{\ensuremath{\mathbb{C}}}
\newcommand{\N}{\ensuremath{\mathbb{N}}}
\newcommand{\Q}{\ensuremath{\mathbb{Q}}}
\renewcommand{\P}{\ensuremath{\mathcal{P}}}
\newcommand{\Z}{\ensuremath{\mathbb{Z}}}
\newcommand{\kv}{{k^{\vee}}}
\renewcommand{\l}{\lambda}
\newcommand{\gb}{\overline{\mathfrak{g}}}
\newcommand{\g}{\mathfrak{g}}
\newcommand{\gh}{\widehat{\mathfrak{g}}}
\newcommand{\ghN}{\widehat{\mathfrak{g}_{(N)}}}
\newcommand{\gbN}{\overline{\mathfrak{g}_{(N)}}}
\newcommand{\tr}{\mathrm{tr}}
\newcommand{\sln}{\mathfrak{sl}}
\newcommand{\sn}{\mathfrak{s}}
\newcommand{\so}{\mathfrak{so}}
\newcommand{\spn}{\mathfrak{sp}}
\newcommand{\tsp}{\mathfrak{tsp}(2n)}
\newcommand{\gl}{\mathfrak{gl}}
\newcommand{\slnb}{{\overline{\mathfrak{sl}}}}
\newcommand{\snb}{{\overline{\mathfrak{s}}}}
\newcommand{\sob}{{\overline{\mathfrak{so}}}}
\newcommand{\spnb}{{\overline{\mathfrak{sp}}}}
\newcommand{\glb}{{\overline{\mathfrak{gl}}}}
\newcommand{\Hwft}{\mathcal{H}_{F,\tau}}
\newcommand{\Hwftm}{\mathcal{H}_{F,\tau}^{(m)}}

\newcommand{\car}{{\mathfrak{h}}}    
\newcommand{\bor}{{\mathfrak{b}}}    
\newcommand{\nil}{{\mathfrak{n}}}    
\newcommand{\vp}{{\varphi}}
\newcommand{\bh}{\widehat{\mathfrak{b}}}  
\newcommand{\bb}{\overline{\mathfrak{b}}}  
\newcommand{\Vh}{\widehat{\mathcal V}}
\newcommand{\KZ}{Kniz\-hnik-Zamo\-lod\-chi\-kov}
\newcommand{\TUY}{Tsuchia, Ueno  and Yamada}
\newcommand{\KN} {Kri\-che\-ver-Novi\-kov}
\newcommand{\pN}{\ensuremath{(P_1,P_2,\ldots,P_N)}}
\newcommand{\xN}{\ensuremath{(\xi_1,\xi_2,\ldots,\xi_N)}}
\newcommand{\lN}{\ensuremath{(\lambda_1,\lambda_2,\ldots,\lambda_N)}}
\newcommand{\iN}{\ensuremath{1,\ldots, N}}
\newcommand{\iNf}{\ensuremath{1,\ldots, N,\infty}}

\newcommand{\tb}{\tilde \beta}
\newcommand{\tk}{\tilde \kappa}
\newcommand{\ka}{\kappa}
\renewcommand{\k}{\kappa}

\newcommand{\Pif} {P_{\infty}}
\newcommand{\Pinf} {P_{\infty}}
\newcommand{\PN}{\ensuremath{\{P_1,P_2,\ldots,P_N\}}}
\newcommand{\PNi}{\ensuremath{\{P_1,P_2,\ldots,P_N,P_\infty\}}}
\newcommand{\Fln}[1][n]{F_{#1}^\lambda}
\newcommand{\tang}{\mathrm{T}}
\newcommand{\Kl}[1][\lambda]{\can^{#1}}
\newcommand{\A}{\mathcal{A}}
\newcommand{\U}{\mathcal{U}}
\newcommand{\V}{\mathcal{V}}
\renewcommand{\O}{\mathcal{O}}
\newcommand{\Ae}{\widehat{\mathcal{A}}}
\newcommand{\Ah}{\widehat{\mathcal{A}}}
\newcommand{\La}{\mathcal{L}}
\newcommand{\Le}{\widehat{\mathcal{L}}}
\newcommand{\Lh}{\widehat{\mathcal{L}}}
\newcommand{\eh}{\widehat{e}}
\newcommand{\Da}{\mathcal{D}}
\newcommand{\kndual}[2]{\langle #1,#2\rangle}
\newcommand{\cins}{\frac 1{2\pi\mathrm{i}}\int_{C_S}}
\newcommand{\cinsl}{\frac 1{24\pi\mathrm{i}}\int_{C_S}}
\newcommand{\cinc}[1]{\frac 1{2\pi\mathrm{i}}\int_{#1}}
\newcommand{\cintl}[1]{\frac 1{24\pi\mathrm{i}}\int_{#1 }}
\newcommand{\w}{\omega}
\newcommand{\ord}{\operatorname{ord}}
\newcommand{\res}{\operatorname{res}}
\newcommand{\nord}[1]{:\mkern-5mu{#1}\mkern-5mu:}
\newcommand{\Fn}[1][\lambda]{\mathcal{F}^{#1}}
\newcommand{\Fl}[1][\lambda]{\mathcal{F}^{#1}}
\renewcommand{\Re}{\mathrm{Re}}

\newcommand{\ha}{H^\alpha}

\define\ldot{\hskip 1pt.\hskip 1pt}
\define\ifft{\qquad\text{if and only if}\qquad}
\define\a{\alpha}
\redefine\d{\delta}
\define\w{\omega}
\define\ep{\epsilon}
\redefine\b{\beta} \redefine\t{\tau} \redefine\i{{\,\mathrm{i}}\,}
\define\ga{\gamma}
\define\cint #1{\frac 1{2\pi\i}\int_{C_{#1}}}
\define\cintta{\frac 1{2\pi\i}\int_{C_{\tau}}}
\define\cintt{\frac 1{2\pi\i}\oint_{C}}
\define\cinttp{\frac 1{2\pi\i}\int_{C_{\tau'}}}
\define\cinto{\frac 1{2\pi\i}\int_{C_{0}}}
\define\cinttt{\frac 1{24\pi\i}\int_C}
\define\cintd{\frac 1{(2\pi \i)^2}\iint\limits_{C_{\tau}\,C_{\tau'}}}
\define\cintdr{\frac 1{(2\pi \i)^3}\int_{C_{\tau}}\int_{C_{\tau'}}
\int_{C_{\tau''}}}
\define\im{\operatorname{Im}}
\define\re{\operatorname{Re}}
\define\res{\operatorname{res}}
\redefine\deg{\operatornamewithlimits{deg}}
\define\ord{\operatorname{ord}}
\define\rank{\operatorname{rank}}
\define\fpz{\frac {d }{dz}}
\define\dzl{\,{dz}^\l}
\define\pfz#1{\frac {d#1}{dz}}

\define\K{\Cal K}
\define\U{\Cal U}
\redefine\O{\Cal O}
\define\He{\text{\rm H}^1}
\redefine\H{{\mathrm{H}}}
\define\Ho{\text{\rm H}^0}
\define\A{\Cal A}
\define\Do{\Cal D^{1}}
\define\Dh{\widehat{\mathcal{D}}^{1}}
\redefine\L{\Cal L}
\newcommand{\ND}{\ensuremath{\mathcal{N}^D}}
\redefine\D{\Cal D^{1}}
\define\KN {Kri\-che\-ver-Novi\-kov}
\define\Pif {{P_{\infty}}}
\define\Uif {{U_{\infty}}}
\define\Uifs {{U_{\infty}^*}}
\define\KM {Kac-Moody}
\define\Fln{\Cal F^\lambda_n}
\define\gb{\overline{\mathfrak{ g}}}
\define\G{\overline{\mathfrak{ g}}}
\define\Gb{\overline{\mathfrak{ g}}}
\redefine\g{\mathfrak{ g}}
\define\Gh{\widehat{\mathfrak{ g}}}
\define\gh{\widehat{\mathfrak{ g}}}
\define\Ah{\widehat{\Cal A}}
\define\Lh{\widehat{\Cal L}}
\define\Ugh{\Cal U(\Gh)}
\define\Xh{\hat X}
\define\Tld{...}
\define\iN{i=1,\ldots,N}
\define\iNi{i=1,\ldots,N,\infty}
\define\pN{p=1,\ldots,N}
\define\pNi{p=1,\ldots,N,\infty}
\define\de{\delta}

\define\kndual#1#2{\langle #1,#2\rangle}
\define \nord #1{:\mkern-5mu{#1}\mkern-5mu:}
\define \sinf{{\widehat{\sigma}}_\infty}
\define\Wt{\widetilde{W}}
\define\St{\widetilde{S}}
\define\Wn{W^{(1)}}
\define\Wtn{\widetilde{W}^{(1)}}
\define\btn{\tilde b^{(1)}}
\define\bt{\tilde b}
\define\bn{b^{(1)}}
%
\define\eps{\varepsilon}    
\define\doint{({\frac 1{2\pi\i}})^2\oint\limits _{C_0}
       \oint\limits _{C_0}}                            
\define\noint{ {\frac 1{2\pi\i}} \oint}   
\define \fh{{\frak h}}     
\define \fg{{\frak g}}     
\define \GKN{{\Cal G}}   
\define \gaff{{\hat\frak g}}   
\define\V{\Cal V}
\define \ms{{\Cal M}_{g,N}} 
\define \mse{{\Cal M}_{g,N+1}} 
\define \tOmega{\Tilde\Omega}
\define \tw{\Tilde\omega}
\define \hw{\hat\omega}
\define \s{\sigma}
\define \car{{\frak h}}    
\define \bor{{\frak b}}    
\define \nil{{\frak n}}    
\define \vp{{\varphi}}
\define\bh{\widehat{\frak b}}  
\define\bb{\overline{\frak b}}  
\define\Vh{\widehat V}
\define\KZ{Knizhnik-Zamolodchikov}
\define\ai{{\alpha(i)}}
\define\ak{{\alpha(k)}}
\define\aj{{\alpha(j)}}
\newcommand{\laxgl}{\overline{\mathfrak{gl}}}
\newcommand{\laxsl}{\overline{\mathfrak{sl}}}
\newcommand{\laxso}{\overline{\mathfrak{so}}}
\newcommand{\laxsp}{\overline{\mathfrak{sp}}}
\newcommand{\laxs}{\overline{\mathfrak{s}}}
\newcommand{\laxg}{\overline{\frak g}}
\newcommand{\bgl}{\laxgl(n)}
\newcommand{\tX}{\widetilde{X}}
\newcommand{\tY}{\widetilde{Y}}
\newcommand{\tZ}{\widetilde{Z}}

\vspace*{-1cm}
%
%
%
\vspace*{2cm}

\title[Hamiltonian integrable hierarchies]
{Lax operator algebras and Hamiltonian integrable hierarchies}
\author[O.K. Sheinman]{Oleg K. Sheinman}
\thanks{Supported in part by the RFBR project 08-01-00054-a and
by the program "Mathematical methods of Nonlinear Dynamics" of the
Russian Academy of Science}
\address[ Oleg K. Sheinman]{Steklov Mathematical Institute, ul. Gubkina, 8,
Moscow 119991, Russia, and Independent University of Moscow,
Bolshoi Vlasievskii per. 11, Moscow, Russia}
\email{sheinman@mi.ras.ru}


\begin{abstract}
We consider the theory of Lax equations in complex simple and
reductive classical Lie algebras with the spectral parameter on a
Riemann surface of finite genus. Our approach is based on the new
objects --- the Lax operator algebras, and develops the approach
of I.Krichever treating the $\gl(n)$ case. For every Lax operator
considered as the mapping sending a point of the cotangent bundle
on the space of extended Tyrin data to an element of the
corresponding Lax operator algebra we construct the hierarchy of
mutually commuting flows given by Lax equations and prove that
those are Hamiltonian with respect to the Krichever-Phong
symplectic structure. The corresponding Hamiltonians give
integrable finite-dimensional Hitchin-type systems. For example we
derive elliptic $A_n$, $C_n$, $D_n$ Calogero-Moser systems in
frame of our approach.
\end{abstract} 
\maketitle
\tableofcontents
\section{Introduction}\label{S:intro}

In \cite{rKNU,KNfaa} I.M.Krichever and S.P.Novikov proposed the
technique of finding high rank finite-zone solutions to
Kadomtsev-Petviashvili and Shr\"{o}dinger equations. Based on the
ideas of those works and on his results on effective
classification of high rank pairs of commuting differential
operators \cite{KrComm} I.M.Krichever proposed the theory of Lax
operators with the spectral parameter on a Riemann surface
\cite{Klax}. In \cite{KSlax} I.M.Kricher and the author found that
these Lax operators form an associative algebra, and constructed
their orthogonal and symplectic analogs which form Lie algebras.
They were called {\it Lax operator algebras}. Lax operator
algebras form a new class of one-dimensional current algebras.

The applications of current algebras to the theory of Lax
equations have a long history. They are initiated in the works of
I.Gelfand, L.Dikii, I.Dorfman, A.Reyman, M.Semenov-Tian-Shanskii,
V.Drinfeld, V.Sokolov, V.Kac, P. van Moerbeke. Basically these
applications are related to Kac-Moody algebras which appear quite
naturally in the context of Lax equations with {\it rational}
spectral parameter. In \cite{Klax} the theory of conventional Lax
and zero curvature representations with rational spectral
parameter was generalized to the case of algebraic curves $\Sigma$
of arbitrary genus $g$. Such representations arise in several ways
in the theory of integrable systems, c.f. \cite{rKNU} where a zero
curvature representation of the Krichever-Novikov equation is
introduced, or \cite{Klax} where a field analog of the
Calogero-Moser system on an elliptic curve is presented. Lax
operator algebras appear as a corresponding generalization of
Kac-Moody algebras.

In \cite{ShN70AMS,Sh_lopa} we have posed the problem of
generalization of the Krichever theory \cite{Klax} to all Lax
operator algebras, and have done the first steps in that direction
including the construction of integrable hierarchies of Lax
equations. In the present article we undertake the concluding
step: prove that those Lax equations are Hamiltonian and construct
the corresponding Hamiltonians. We consider certain wellknown
examples (the elliptic Calogero-Moser systems) in the context of
our approach.

The concept of Lax operators on algebraic curves is closely
related to A.Tyurin results on the classification of holomorphic
vector bundles on algebraic curves \cite{Tyvb}. It uses {\it
Tyurin data} modelled on {\it Tyurin parameters} of such bundles.
Turin data consist of points $\ga_s\in\Sigma$ ($s=1,\ldots ,ng$),
and associated elements $\a_s\in\C P^n$ where $g$ denotes the
genus of the Riemann surface $\Sigma$, and $n$ corresponds to the
rank of the bundle. Every Lax operator algebra is associated with
fixed Tyurin data and a set of marked points on $\Sigma$. For a
finite-dimensional simple or reductive Lie algebra $\g$ over $\C$
we denote this algebra by~$\gb$. It is an almost graded Lie
algebra. Its elements are parameterized by cotangent vectors to
the space of Tyurin data at the point. Those elements are defined
in \cite{Klax,KSlax} as meromorphic $(n\times n)$ matrix-valued
functions on $\Sigma$ having arbitrary poles at the points $P_k$,
$k=1,\ldots,N$ (which are assumed to be fixed), and poles of the
order at most two at $\gamma_s$'s. The coefficients of the Laurent
expansion of those matrix-valued functions in the neighborhood of
a point $\gamma_s$ have to obey certain constraints parametrized
by $\alpha_s$ (relations \refE{triv} below). In the case of
absence of the points $\ga_s$ (which corresponds to trivial vector
bundles) we return to the known class of Krichever-Novikov
algebras (see \cite{ShN65} for a review). If, in addition, the
genus of $\Sigma$ is equal to $0$, $\sharp\{P_i\}=2$, and these
two points are $0$ and $\infty$ we obtain the loop algebras. It
may be mentioned here that all these three types of current
algebras have quite similar theory of central extensions
\cite{KSlax,SSlax}

By Lax operator we mean the mapping sending a point of the
cotangent bundle on the space of Tyrin data (extended by marked
points) to an element of the corresponding Lax operator algebra.
With every positive divisor $D=\sum m_iP_i$ we associate the
subspace $\L^D$ of the cotangent bundle such that the Lax
operator, as a function on $\Sigma$, has a pole of order at most
$m_i$ at $P_i$ for every fixed point in that subspace. We consider
a certain dynamical system on $\L^D$ given by the Lax equation. We
prove that these dynamical systems are Hamiltonian and integrable
with respect to the Krichever-Phong symplectic structure.

In \refS{Moper} following \cite{Klax,Sh_lopa} we introduce {\it
$M$-operators} (partners of Lax operators in Lax pairs) and study
their analytic and algebraic properties. Every $M$-operator
defines a flow in the cotangent bundle on the space of extended
Tyurin data. The corresponding dynamics of the Tyurin parameters
\refE{mdata} plays a fundamental role. It appeared in \cite{rKNU}
as an ansatz for effective integration of KP equation in the class
of high rank solutions, and is hardly used in \cite{Klax}.

In \refS{algebras} we give a brief survey of Lax operator
algebras. Following \cite{Klax,Sh_lopa} we introduce their
elements ({\it $L$-operators}) as $M$-operators yielding a trivial
dynamics of Tyurin data. There is also an independent definition
of $L$-operators given in \cite{KSlax}. We refer to that article
for more detail.

In \refS{Lax} we prove that a Lax equation is equal to the system
of equations on the main parts of the Lax operator at the points
$P_i$, and of the above mentioned dynamics equations of Tyurin
data. We prove a criteria of well-definiteness of the dynamical
system given by a Lax equation on a space $\L^D$.

In \refS{Hierarch} we construct hierarchies of commuting flows on
the cotangent bundle on $\L^D$ given by a Lax operator. A key role
is played by the \refL{dimND} which treats the dimension of the
space of $M$-operators satisfying the criteria of \refS{Lax}.
Modulo that lemma we follow the lines of \cite{Klax}. We prove the
final result for all classical Lie algebras in question, in
particular for $\spn(2n)$ which was not done in the earlier
authors article \cite{Sh_lopa} on hierarchies. We would like to
stress here that in general $L$ and $M$ do not belong to the same
Lie algebra and give the list of the corresponding pairs of Lie
algebras \refE{corresp}. Observe that the method of constructing a
hierarchy given in this section is not unique. In \refS{CMoser} we
give another one for elliptic curves enabling us to incorporate
certain Calogero-Moser systems in frame of our approach.

In Sections \ref{S:Sympl},\ref{S:Hamilton} we develop a
Hamiltonian theory for the Lax equations in question. In
\refS{Sympl} we construct the analog of the Krichever-Phong
symplectic structure on the cotangent bundle of a certain subspace
$\P^D\subset\L^D/G$ (where $\g=Lie\, G$) invariant with respect to
the flows of our commuting hierarchy. In \refS{Hamilton} we prove
that the hierarchies of \refS{Hierarch} are Hamiltonian with
respect to that structure and construct the corresponding
Hamiltonians. Again, the presentation is similar for all classical
Lie algebras in question, and follows the lines of \cite{Klax}
except for the proofs of non-degenerateness of the Krichever-Phong
symplectic form (\refS{Sympl}) and of holomorphy of spectra of
operators in Lax pairs (\refS{Hamilton}). The latter is done for
every class of Lie algebras individually. We also give the proofs
of the statements in \cite{Klax} in slightly more detail.

In \refS{CMoser} we consider three examples of integrable systems
in frame of our approach. The first one is the elliptic
Calogero-Moser system for $\g=\gl(n)$. It is considered earlier in
\cite{Klax,Kr_ellCM}. Further on, by modifying the technique of
\refS{Hierarch} we obtain the Lax operator and the corresponding
Hamiltonian integrable hierarchy for the elliptic Calogero-Moser
$\so(2n)$ and $\spn(2n)$ systems. Unfortunately we don't know any
explicit form of Lax operators for $\so(2n+1)$ systems in frame of
our approach. In particular, such Lax operators must be
meromorphic in the spectral parameter. Observe that the Lax
representations of those systems with the Lax operators of
Backer-Akhieser type are known \cite{dHPh}.

Author thanks I.Krichever and M.Schlichenmaier for numerous
fruitful discussions. I am also thankful to the audience of my
lecture course in the Independent University of Moscow and to the
members of A.G.Sergeev's seminar in the Steklov mathematical
institute for valuable remarks.

\section{$M$-operators and times}\label{S:Moper}
$ $

Let $\Sigma$ be a compact Riemann surface of genus $g$ with two
marked points $P_+,P_-$, and $\g$ be a Lie algebra over $\C$ from
the following list: $\gl(n)$, $\sln(n)$, $\so(2n)$, $\so(2n+1)$,
$\spn(2n)$, $\sn(n)$, $\tsp$ where $\sn(n)$ is the algebra of
scalar matrices, and $\tsp$ is a Lie subalgebra of $\spn(2n+2)$
consisting of the matrices having zero first column and last row.
This is the list of the previous work \cite{Sh_lopa} extended with
$\tsp$ which appears in the construction of integrable hierarchies
in the symplectic case.

Let us fix $K$ additional  points $\ga_s\in\Sigma$, and let
\begin{equation}
W:=\{\ga_s\in\Sigma\setminus\{P_+,P_-\}\mid s=1,\ldots, K\}
\end{equation}
($K$ will be specified in \refS{Hierarch}). To every point $\ga_s$
we assign a vector $\a_s\in\C^p$ given up to a scalar factor where
$p$ is the dimension of the standard (vector) representation of
the corresponding $\g$ (i.e. the elements of $\g$ are $p\times
p$-matrices). The system
\begin{equation}
T:=\{(\ga_s,\a_s)\mid s=1,\ldots, K\}
\end{equation}
is called {\it Tyurin data} below. This data is related to the
moduli of holomorphic vector bundles over $\Sigma$. In particular,
for generic values of $(\ga_s,\a_s)$ with $\a_s\ne 0$ and $K=ng$
Tyurin data parameterize the semistable rank $n$ degree $n g$
framed holomorphic vector bundles over $\Sigma$, see \cite{Tyvb}.

Let $M:\,\Sigma\to\g$ be a meromorphic function. We require that
at a point $\ga=\ga_s$ $M$ has the expansion
\begin{equation}\label{E:expan}
  M=\frac{M_{-2}}{(z-z_\ga)^2}+\frac{M_{-1}}{z-z_\ga}+M_0+\ldots\ ,
\end{equation}
where $z$ is a fixed local coordinate in the neighborhood of
$\ga$, $z_\ga$ is the coordinate of $\ga$ itself,
$M_{-2},M_{-1},M_0,M_1,\ldots\in\g$ and
\begin{equation}\label{E:res1}
    M_{-2}=\l\a\a^t\s,\ \ \ \ \
  M_{-1}=(\a\mu^t+\eps\mu\a^t)\s
\end{equation}
where $\l\in\C$, $\mu\in\C^n$, $\s$ is a $n\times n$ matrix, the
upper $t$ denotes the matrix transposition
\begin{equation}\label{E:laeps}
  \begin{aligned}
     \l\equiv 0,\ &\eps=0,\ \ \  \s=id\ \ \text{for}\ \g=\gl(n),\sln(n),\\
     \l\equiv 0,\ &\eps=-1,\ \s=id\ \ \text{for}\ \g=\so(n), \\
                  &\eps=1\ \ \ \ \ \ \ \ \ \ \ \ \ \ \
                  \text{for}\ \g=\spn(2n),
  \end{aligned}
\end{equation}
and $\s$ is a matrix of the symplectic form for $\g=\spn(2n)$. In
the case $\g=\tsp$ we require that $\a$ and $\mu$ were of the form
$\a=(\a_0,\tilde\a^t,0)^t$ ($\a_0\in\C$, $\tilde\a\in\C^{2n}$),
$\mu=(\mu_0,\tilde\mu^t,0)^t$ ($\mu_0\in\C$,
$\tilde\mu\in\C^{2n}$), and $\s=
  \left(\begin{smallmatrix}
    0 & 0 & 1 \\
    0 & \boxed{\tilde\s} & 0   \\
    -1 & 0        & 0
  \end{smallmatrix}\right)
$
where $\tilde\s$ is the matrix of the symplectic form
corresponding to $\spn(2n)$. For such $\s$ the elements of $\tsp$
are of the form
$
  X=\left(\begin{smallmatrix}
    0 & u^t & c \\
    \begin{smallmatrix} 0_{_{_{_{}}}}\end{smallmatrix} & \boxed{\tilde X} & \begin{smallmatrix} v_{_{_{_{}}}} \end{smallmatrix}  \\
    0 & 0        & 0
  \end{smallmatrix}\right)
$ where ${\tilde X}\in\spn(2n)$, $u^t=(a^t,b^t)$,
$v^t=(b^t,-a^t)$, and $a,b\in\C^n$. In other words, $\tsp$ is the
semidirect sum of the $\spn(2n)$ and the Heisenberg
algebra\label{Heis}. Here and below we omit the subscripts $s,\ga$
indicating the point $\ga$ except for $z_\ga$. In the cases
$\g=\spn(2n)$, $\g=\tsp$ we require that
\begin{equation}\label{E:add2m}
  \a^t\s M_1\a=0.
\end{equation}
Under above requirements we call $M$ a {\it $\g$-valued
$M$-operator}. In \cite{Sh_lopa} we did not require \refE{add2m}
for $M$-operators doing that for $L$-operators only (see
\refS{algebras}).

Every $M$-operator and a complex number $\k$ define a dynamical
system on the space of Tyurin data:
\begin{equation}\label{E:mdata}
{\dot z_\ga}=-\mu^t\s\a, \ \ {\dot\a}=-M_0\a+\k\a ,
\end{equation}
where the upper dot means the time derivative, $\k\in\C$. We
comment on these equations in \refL{mov-eq} and subsequent
remarks.

\begin{lemma}\label{L:algM} For any two $M$-operators $M_a$, $M_b$
and corresponding times the expression
\[ M_{ab}= \partial_aM_b-\partial_bM_a+[M_a,M_b]
\]
is an $M$-operator too.
\end{lemma}
For all Lie algebras from our list the lemma is proved in
\cite{Sh_lopa}, except for $\g=\spn(2n),\tsp$. We reproduce the
proof of that work in order to show the place for additional
arguments needed in the case $\g=\tsp$. We prove also that
$M_{ab}$ satisfies the relation \refE{add2m} which was not needed
in the setting of \cite{Sh_lopa}.
\begin{proof}

Let us verify that $M_{ab}$ satisfies \refE{res1}.

For an arbitrary $\g$ from our list we have
\[
  M_a=\frac{\l_a\a\a^t\s}{(z-z_\ga)^2}+\frac{(\a\mu_a^t+\eps\mu_a\a^t)\s}{z-z_\ga}+M_{0a}+\ldots\ .
\]
and similar expression for $M_b$ where $\l_a, \l_b$, $\eps$ and
$\s=id$ are subjected to \refE{laeps}. Next we have
\begin{equation}\label{E:damb}
 \begin{aligned}
 \partial_aM_b&=2(\partial_az_\ga)\frac{\l_b\a\a^t\s}{(z-z_\ga)^3}+
 \frac{((\partial_a\l_b)\a\a^t+\l_b\partial_a(\a\a^t))\s
 +(\partial_az_\ga)M_{-1,b}}{(z-z_\ga)^2}+ \\
 &+\frac{((\partial_a\a)\mu_b^t+\eps\mu_b(\partial_a\a^t)+\a(\partial_a\mu_b^t)
 +\eps(\partial_a\mu_b)\a^t)\s}{z-z_\ga}+
 \ldots
 \end{aligned}
\end{equation}
and similar expression for $\partial_bM_a$.

For the commutator we have
\begin{equation}\label{E:mcomm}
 \begin{aligned}
   &[M_a,M_b]=\frac{(1+\eps^2)(\l_b\cdot\mu_a^t\s\a-\l_a\cdot\mu_b^t\s\a)\a\a^t\s}{(z-z_\ga)^3}\\
      &+\frac{(\l_a\partial_b-\l_b\partial_a)\a\a^t\s+\l_{ab}\a\a^t\s
      +(\mu_a^t\s\a)M_{-1,b}-(\mu_b^t\s\a)M_{-1,a}}{(z-z_\ga)^2}+\\
      &+\frac{((\partial_b\a)\mu_a^t+\eps\mu_a(\partial_b\a^t))\s-
        ((\partial_a\a)\mu_b^t+\eps\mu_b(\partial_a\a^t))\s}
            {z-z_\ga}\\
            &+\frac{(\a\mu_{ab}^t+\eps\mu_{ab}\a^t)\s}{z-z_\ga}+\ldots
 \end{aligned}
\end{equation}
where
$\l_{ab}=2\l_b\k_a-2\l_a\k_b+\eps(\mu_a^t\s\mu_b-\mu_b^t\s\mu_a)$,
$\mu_{ab}=\k_a\mu_b-\k_b\mu_a-\l_aM_{1b}\a+\l_bM_{1a}\a-M_{0b}\mu_a+M_{0a}\mu_b$.
To obtain this relation we used the equations \refE{mdata} and
some additional relations, in particular
$\eps\a^t\s\mu=-\eps^2\mu^t\s\a$ and $\l\a^t\s\a=0$ which are
fulfilled in all cases. In the computation of $[M_a,M_b]_{-2}$ we
also used the relation
\[
  [M_{-1,a},M_{-1,b}]=(\mu_a^t\s\a)M_{-1,b}-(\mu_b^t\s\a)M_{-1,a}+
  \eps(\mu_a^t\s\mu_b-\mu_b^t\s\mu_a)\a\a^t\s
\]
which can be verified using \refE{res1}. To obtain
$[M_a,M_b]_{-1}$ in the form \refE{mcomm} it is heavily used that
$M_{i,a}^t=-\s M_{i,a}\s^{-1}$ for $\eps\ne 0$ (which follows from
\refE{laeps}), and the same for~$M_{i,b}$.

In the case $\g=\tsp$ we need to prove that $\mu_{ab}$ has a zero
last coordinate. This is obviously the case because the matrices
$M_{1a}$, $M_{1b}$, $M_{0a}$, $M_{0b}$ have zero last rows.

Comparing \refE{mcomm} and \refE{damb} (and the corresponding
relation for $\partial_bM_a$) and using \refE{mdata} we obtain
\[
  M_{ab}=\frac{{\tilde\l_{ab}}\a\a^t\s}{(z-z_\ga)^2}
  +\frac{\left(\a{\tilde\mu_{ab}^t}
 +\eps{\tilde\mu_{ab}}\a^t\right)\s}{z-z_\ga}+\ldots
\]
where $\tilde\l_{ab}=\partial_a\l_b-\partial_b\l_a+\l_{ab}$,
$\tilde\mu_{ab}=\partial_a\mu_b-\partial_b\mu_a+\mu_{ab}$. We
observe that $M_{ab}$ has the form \refE{expan}, \refE{res1}. In
particular the $-3$ order term vanishes because either
$\l_a=\l_b=0$ or $\eps^2=1$ (which follows from \refE{laeps}).

Let us prove that $M_{ab}$ satisfies \refE{add2m} in the
symplectic case. We have
\begin{equation}\label{E:comp1}
(\partial_aM_b)_1=\partial_a M_{1b}-2(\partial_az_\ga)M_{2b},\
(\partial_bM_a)_1=\partial_b M_{1a}-2(\partial_bz_\ga)M_{2a}.
\end{equation}
Applying $\partial_a$ to both parts of the relation $\a^t\s
M_{1,b}\a=0$, and $\partial_b$ to the corresponding relation for
$M_{1a}$, and using \refE{mdata} we obtain
\begin{equation}\label{E:comp2}
\a^t\s(\partial_a M_{b1})\a=-\a^t\s[M_{0a},M_{1b}]\a,\
   \a^t\s(\partial_b M_{a1})\a=-\a^t\s[M_{0b},M_{1a}]\a .
\end{equation}
Further on, we have
  \begin{align*}
   [M_a,M_b]_1 &= \l_a[\a\a^t\s,M_{3b}]+[(\a\mu_a^t+\mu_a\a^t)\s
   ,M_{2b}]+[M_{0a},M_{1b}] \\
   &+ [M_{1a},M_{0b}]+
   [M_{a2},(\a\mu_b^t+\mu_b\a^t)\s]+\l_b[M_{3a},\a\a^t\s].
\end{align*}
After taking $\a^t\s [M_a,M_b]_1\a$ the first and the last
commutators in the last relation vanish by $\a^t\s\a=0$ (as well
as some terms of the second and the fifth commutators). The two
commutators in the middle annihilate with the commutators on the
right hand sides of \refE{comp2}. The remaining terms of the
second and the fifth commutators give
\begin{align*}
  &\a^t\s(\mu_a\a^t\s
  M_{2b}-M_{2b}\a\mu_a^t\s+M_{2a}\a\mu_b^t\s-\mu_b\a^t\s M_{2a})\a=\\
  &=2\a^t\s\mu_a(\a^t\s M_{2b}\a)-2\a^t\s\mu_b(\a^t\s M_{2a}\a)
\end{align*}
which annihilate with the corresponding terms coming from
\refE{comp1} due to \refE{mdata}.
\end{proof}

\section{$L$-operators and Lax operator algebras}\label{S:algebras}


We define $L$-operators as $M$-operators yielding trivial dynamics
by \refE{mdata}. Thus, by definition, every $L$-operator $L$ is a
meromorphic $\g$-valued function on $\Sigma$ holomorphic outside
$W\cup \{P_+, P_-\}$ such that at a point $\ga=\ga_s$
\begin{equation}\label{E:lexpan}
  L=\frac{L_{-2}}{(z-z_\ga)^2}+\frac{L_{-1}}{z-z_\ga}+L_0+\ldots\ ,
\end{equation}
where $z$ is a fixed local coordinate in the neighborhood of
$\ga$, $z_\ga$ is the coordinate of $\ga$ itself,
$L_{-2},L_{-1},L_0,L_1,\ldots\in\g$ and
\begin{equation}\label{E:lres1}
    L_{-2}=\nu\a\a^t\s,\ \ \ \ \
  L_{-1}=(\a\b^t+\eps\b\a^t)\s
\end{equation}
where $\nu\in\C$, $\b\in\C^n$, $\s$ is a $n\times n$ matrix,
\begin{equation}\label{E:llaeps}
  \begin{aligned}
     \nu\equiv 0,\ &\eps=0,\ \ \  \s=id\ \ \text{for}\ \g=\gl(n),\sln(n),\\
     \nu\equiv 0,\ &\eps=-1,\ \s=id\ \ \text{for}\ \g=\so(n), \\
                  &\eps=1\ \ \ \ \ \ \ \ \ \ \ \ \ \ \
                  \text{for}\ \g=\spn(2n),
  \end{aligned}
\end{equation}
and $\s$ is a matrix of the symplectic form for $\g=\spn(2n)$. The
case $\g=\tsp$ is not considered in this section; nevertheless,
such consideration is possible.

Further on, the requirement of triviality of the dynamics
\refE{mdata} writes as
\begin{equation}\label{E:triv}
\b^t\s\a=0, \ \ L_0\a=\k\a .
\end{equation}
In addition we assume that
\begin{equation}\label{E:add1}
\a^t\a=0\ \ \text{for}\ \ \g=\so(n)
\end{equation}
and
\begin{equation}\label{E:add2} \a^t\s L_1\a=0\ \ \text{for}\ \
\g=\spn(2n).
\end{equation}
\begin{theorem}\label{T:struct}
The space $\gb$ of L-operators  is a Lie algebra under the
point-wise matrix commutator. For $\g=\gl(n)$ it is an associative
algebra under the point-wise matrix multiplication as well.
\end{theorem}
The theorem is proven in \cite{KSlax}. Another proof is given in
\cite{Sh_lopa} where the theorem is derived from the \refL{algM}.
It can be easy proven also for $\g=\tsp$. In the last case, making
use of the already proven result in the symplectic case, we obtain
that given $L_1,L_2\in\overline{{\mathfrak{sp}}(2n+2)}$, the
commutator $[L_1,L_2]$ belongs to
$\overline{{\mathfrak{sp}}(2n+2)}$ too. If $L_1,L_2\in\tsp$ at
every point then $[L_1,L_2]$ does as well, and according to the
proof of the \refL{algM} the corresponding $\b$ has a zero last
coordinate. Hence $[L_1,L_2]\in\overline\tsp$.

The Lie algebra $\gb$ is called a {\it Lax operator algebra}. It
$\gb$ depends both on the choice of Tyurin parameters and of the
points $P_+$ and $P_-$ but we omit any indication on this
dependence in our notation.

Consider $\glb(n)$ in more detail.

In this case $L_{-2}=0$, $L_{-1}=\a\b^t$ where $\b^t\a=0$ and
$L_0\a=\k\a$. These constraints imply that the elements of the Lax
operator algebra $\glb(n)$ can be considered as sections of the
endomorphism bundle $End(B)$, where $B$ is the holomorphic vector
bundle corresponding to the Tyurin data $T$.

The splitting $\gl(n)=\sn(n)\oplus \sln(n)$ given by
\begin{equation}
X\mapsto \left(\ \frac {\tr(X)}{n}I_n\ ,\ X-\frac {\tr(X)}{n}I_n\
\right)
\end{equation}
(where $I_n$ is the $n\times n$ unit matrix) induces a
corresponding splitting for $\glb(n)$:
\begin{equation}
 \glb(n)=\snb(n)\oplus \slnb(n).
\end{equation}

For $\snb(n)$ all coefficients in \refE{lexpan} are scalar
matrices. For this reason, the coefficients $L_{-1}$ vanish for
all $\ga\in W$, hence the elements of $\snb(n)$ are holomorphic at
$W$. Also $L_{s,0}$, as a scalar matrix, has any $\a_s$ as an
eigenvector. This means that, by definition,
\begin{equation}\label{E:sint} \snb(n)\cong \sn(n)\otimes \A\cong
\A
\end{equation}
as associative algebras.

Any Lax operator algebra $\gb$ possesses an {\it almost-graded
structure} (see \refT{almgrad} below for the definition).

Assume, all our marked points (including the points in $W$) are in
generic position, and $W\ne\emptyset$. Let us choose local
coordinates $z_\pm$ at $P_\pm$, and $z_s$ at $\ga_s$, $s=1,\ldots,
K$. Assume $\g$ to be a simple Lie algebra from our list. For an
arbitrary $m\in\Z$ consider the subspace
\begin{multline}\label{E:almdeg}
\gb_m:=\{L\in\gb\mid \exists X_+,X_-\in\g \quad \text{such that}
\\
L(z_+)=X_+z^m_++O(z_+^{m+1}),\
L(z_-)=X_-z^{-m-g}_-+O(z_-^{-m-g+1})\}.
\end{multline}
For $\g=\gl(n)$ it is proven above that
$\laxgl(n)=\laxsl(n)\oplus\A\cdot id$ where $\A$ is the
Krichever-Novikov function algebra. In this case we set
\begin{equation}
 \laxgl(n)_m=\laxsl(n)_m\oplus\A_m\cdot id
\end{equation}
where $\A_m$ is the corresponding homogeneous subspace for $\A$
\cite{KNFa}. If $W=\emptyset$, we are in the setup of
Krichever-Novikov algebras and use the corresponding prescriptions
\cite{KNFa,ShN65}.

We call $\gb_m$ a ({\it homogeneous}) {\it  subspace of degree
$m$} in $\gb$.
\begin{theorem}[\cite{KSlax}]\label{T:almgrad}
The subspaces $\gb_m$ give the stucture of an almost-graded Lie
algebra on $\gb$. More precisely,

\noindent (1) $\dim \gb_m=\dim \g$;

\noindent (2) $\gb=\bigoplus\limits_{m\in\Z}^{}\gb_m$;

\noindent (3) $ [\gb_m,\gb_k]\subseteq
\bigoplus\limits_{h=m+k}^{m+k+M}\gb_h$, \newline where $M=g$ for
$\laxsl(n)$, $\laxso(n)$, $\laxsp(2n)$, $M=g+1$ for $\laxgl(n)$.
\end{theorem}
\begin{corollary}\label{C:locex}
Let $X$ be an element of $\g$. For each $m$  there is a unique
element $X_m$ in $\gb_m$ such that
\begin{equation}\label{E:locex}
X_m= Xz_+^m+O(z_+^{m+1}).
\end{equation}
\end{corollary}
\begin{proof}
From the first statement of \refT{almgrad}, i.e. that
 $\dim \gb_m=\dim \g$
it follows that there is a unique combination of the basis
elements such that \refE{locex} is true.
\end{proof}

\section{$\g$-valued Lax equations}
\label{S:Lax}

In this section, we consider consistency of Lax equations of the
form
\begin{equation}\label{E:Lax}
   L_t=[L,M],   \quad L\in \gb , \quad M\in\overline{\g^\diamond}
\end{equation}
where $L,M$ are an $L$-operator and an $M$-operator respectively,
and the correspondence between $\g$ and $\g^\diamond$ is given as
follows:
\begin{equation}\label{E:corresp}
\g^\diamond=  \begin{cases}
    {\gl(n)} & \text{if}\ \g=\gl(n),\sln(n) \\
    {\so(2n+1)} & \text{if}\ \g=\so(2n),\so(2n+1)\\
    {\tsp} & \text{if}\ \g=\spn(2n).
  \end{cases}
\end{equation}
In all cases $\g$ is assumed to be embedded into $\g^\diamond$ in
a standard way, thus the commutator $[L,M]$ is well-defined.

Following \cite{Klax} let us assign every effective divisor
$D=\sum_i m_iP_i$ on $\Sigma$ with the space
$\L^D=\bigcup\limits_{(\a,\ga)}L^D_{\a,\ga}$ (over all Tyurin
parameters $(\a,\ga)$ satisfying \refE{add1}) where
\[ L^D_{\a,\ga}=\{ L\in\gb_{\a,\ga}\ |(L)+D\ge 0 \}
\]
and $\gb_{\a,\ga}$ is the Lax operator algebra corresponding to
$(\a,\ga)$.

Under certain conditions given by the \refL{move} below, the Lax
equation \refE{Lax} gives a well-defined dynamical system on
$\L^D$.

Let the upper dot mean the time derivative, and the term {\it
$\ga$-points} be reserved for the points $\ga_s$.
\begin{lemma}\label{L:mov-eq}
At $\ga$-points, the equations on main parts of $L$ and $M$
following from \refE{Lax} are fulfilled under the following
(sufficient) conditions:
\begin{equation}\label{E:mov-data}
{\dot z_\ga}=-\mu^t\s\a, \ \ {\dot\a}=-M_0\a+\k\a ,
\end{equation}
\begin{equation}\label{E:mov-b}
\begin{aligned} {\dot\b}&= M_0^t\b-L_0^t\mu+\k_L\mu-\k\b \
   \text{for}\ \g=\gl(n), \sln(n),\\
   {\dot\b}&= -M_0\b+L_0\mu+\k_L\mu-\k\b \
       \text{for}\ \g=\so(n),     \\
{\dot\b}&=  -M_0\b+L_0\mu+\k_L\mu-\k\b-\nu M_1\a+\l L_1\a \
       \text{for}\ \g=\spn(2n),
\end{aligned}
\end{equation}
\begin{equation}\label{E:mov-n}
  {\dot\nu}=2(\b^t\s\mu+\l\k_L-\nu\k)\ \text{for}\ \g=\spn(2n)
\end{equation} where $\k_L$ is
defined by $L_0\a=\k_L\a$. Moreover conditions ${\dot
z_\ga}=-\mu^t\s\a$ and \refE{mov-b} are necessary.
\end{lemma}
\begin{proof}
By a straightforward computation we have
\begin{equation}\label{E:dot_L}
\begin{aligned} {\dot L}&= 2{\dot z}_\ga\frac{\nu\a\a^t\s}{(z-z_\ga)^3}+
   \frac{{\dot\nu}\a\a^t\s+\nu\dot\a\a^t\s+\nu\a\dot\a^t\s+{\dot z}_\ga(\a\b^t
   +\eps\b\a^t)\s}{(z-z_\ga)^2}+\\
   &+ \frac{\dot\a\b^t\s+\a\dot\b^t\s+\eps\dot\b\a^t\s+\eps\b\dot\a^t\s}{(z-z_\ga)}
   +({\dot L_0}-{\dot z_\ga}L_1)+\ldots\ .
\end{aligned}
\end{equation}
Using \refE{mcomm} for $M_a=L$, $M_b=M$ we obtain
\begin{equation}\label{E:LMcomm}
\begin{aligned} \left[L,M\right]&= \frac{(1+\eps)^2(-\nu\cdot\mu^t\s\a)\a\a^t\s}{(z-z_\ga)^3}+
   \frac{\nu(\dot\a\a^t+\a\dot\a^t)\s+\l_{ab}\a\a^t\s-(\mu^t\s\a)L_{-1}}{(z-z_\ga)^2}+\\
   &+ \frac{(\dot\a\b^t+\eps\b\dot\a^t)\s+(\a\mu^t_{ab}+\eps\mu_{ab}\a^t)\s}{(z-z_\ga)}
   +\ldots\ .
\end{aligned}
\end{equation}
Note that the second of the relations \refE{mov-data} is used in
deriving the last relation.

If $\nu\ne 0$ (i.e. $\g=\spn(2n)$) then the order $-3$ terms are
equal if and only if
\[ {\dot z}_\ga=-\mu^t\s\a .
\]
If $\nu\equiv 0$ then the order $-3$ terms of \refE{dot_L} and
\refE{LMcomm} both are equal to $0$. The order~$-2$ terms are
equal if and only if
\[ \dot\nu\a\a^t\s+{\dot
z}_\ga(\a\b^t+\eps\b\a^t)\s=\l_{ab}\a\a^t\s-(\mu^t\s\a)L_{-1}
\]
where $\l_{ab}$ is defined by \refE{mcomm}. By the previous
relation we have
\[ \dot\nu\a\a^t\s=\l_{ab}\a\a^t\s
\]
which is fulfilled if $\dot\nu=\l_{ab}$. Note that this relation
is trivial except for $\g=\spn(2n)$, in which case
$\l_{ab}=2(\l\k_L-\nu\k+2\b^t\s\mu)$ and our relation coincides
with \refE{mov-n}.

In a similar way, comparing the $-1$ order terms of the relations
\refE{dot_L} and \refE{LMcomm} we observe that they are equal if
$\dot\b=\mu_{ab}$ where $\mu_{ab}$ are defined by \refE{mcomm}.
This gives the relations \refE{mov-b}. Since the relations
\refE{mcomm} themselves are derived under the assumptions
\refE{mov-data} we obtain the Lemma.
\end{proof}
\begin{remark}
The equation \refE{mov-b} for $\g=\so(n)$ follows from the
corresponding equation for $\gl(n)$ by relations $M_0^t=-M_0$,
$L_0^t=-L_0$. It follows also from the equation for $\g=\spn(2n)$,
with the corresponding replacement of the matrix $\s$, if
$\l=\nu=0$ which is true for $\g=\so(n)$.
\end{remark}
\begin{remark}
 The second condition in \refE{mov-data} and conditions
 \refE{mov-b} are not necessary. The statement remains true if we
 take ${\dot\a}=-M_0\a$ in \refE{mov-data} and exclude the term
 $\k\b$ in \refE{mov-b}.
\end{remark}
\begin{corollary}
 The following relation holds along solutions of the Lax equation
 in the symplectic case regardless to the requirements \refE{add2m} and
 \refE{add2}:
 \[
    \nu\cdot\a^t\s M_1\a=\l\cdot\a^t\s L_1\a .
 \]
\end{corollary}
\begin{proof}
Let us multiply the last of the relations \refE{mov-b} by $\a^t\s$
from the left:
\[
\a^t\s{\dot\b}=  -\a^t\s M_0\b+\a^t\s
L_0\mu+\k_L\a^t\s\mu-\k\a^t\s\b-\nu\a^t\s M_1\a+\l\a^t\s L_1\a
\]
and perform the following replacements: $\a^t\s M_0=\dot\a^t\s$,
$\a^t\s L_0=-\k_L\a^t\s$, $\a^t\s\b=~0$. We will obtain
\[
\nu\a^t\s M_1\a-\l\a^t\s L_1\a=\frac{d}{dt}(\a^t\s\b)=0.
\]
\end{proof}
The next Lemma shows that the equations \refE{mov-b},
\refE{mov-n} can be thrown off. The equations \refE{mov-data} are
most important. These are the  equations of motion for Tyurin
parameters. They are heavily used in \cite{Klax}. Originally, the
concept of moving Tyurin parameters was invented in \cite{rKNU},
where it served for effective solution of Kadomtsev-Petviashvili
equations in certain cases.

Let $T_L\L^D$ denote the tangent space to $\L^D$ at a point $L$.
\begin{lemma}\label{L:move} $[L,M]\in T_L\L^D\quad\Leftrightarrow\quad
([L,M])+D\ge 0$ outside $\ga$'s and the equations \refE{mov-data}
are fulfilled at every $\ga$.
\end{lemma}
\begin{proof} In our proof we follow the lines of \cite{Klax}
where the Lemma was formulated and proved for $\g=\gl(n)$.

Let $z$ be a local coordinate in an open set containing a point
$\ga$, and $z_\ga$ be the corresponding coordinate of $\ga$.

Identify $T\L^D$ with the space ${\mathcal T}^D$ of all
meromorphic $\g$-valued functions $T$ such that at every $\ga$
\begin{equation}\label{E:t-expan}
 \begin{aligned}
 T&=2\dot z_\ga\frac{\nu\a\a^t\s}{(z-z_\ga)^3}+
 \frac{\dot\nu\a\a^t\s+\nu(\dot\a\a^t+\a\dot\a^t)\s
 +{\dot z_\ga}(\a\b^t+\eps\b\a^t)\s}{(z-z_\ga)^2}+ \\
 &+\frac{({\dot\a}\b^t+\eps\b{\dot\a^t}+\a{\dot\b^t}
 +\eps{\dot\b}\a^t)\s}{z-z_\ga}+T_0+ \ldots\ ,
 \end{aligned}
\end{equation}

\begin{equation}\label{E:t-tr}
 {\dot\a}^t\s\b+\a^t\s{\dot\b}=0
\end{equation}
\begin{equation}\label{E:t-eig}
  T_0\a=\k{\dot\a}+{\dot\k}\a-L_0{\dot\a}-{\dot z_\ga}L_1\a
\end{equation}
where ${\dot z_\ga}, {\dot\k}$ are constants, ${\dot\a},{\dot\b}$
are constant vectors fulfilling the relations \refE{mov-data}, and
the divisor of $T$ outside the points $\ga$ is greater or equal to
$-D$.

The relation \refE{t-expan} is modelled on \refE{dot_L} which is
obtained by the time derivation of \refE{lexpan}. In particular
\[ T_0=\dot L_0-\dot z_\ga L_1.
\]
Together with the time derivation of \refE{triv} this gives
\refE{t-eig}. Thus $T\L^D$ embeds to ${\mathcal T}^D$. Let us
check coincidence of dimensions of those spaces. It can be done in
a quite uniform way, so we show it in the most difficult case
$\g=\spn(2n)$. We have $(T)+{\tilde D}\ge 0$ where ${\tilde
D}=D+3\sum\ga$, and $\deg{\tilde D}=\deg\, D+3K$. By Riemann-Roch
theorem  $\dim\{T|(T)+{\tilde D}\ge 0\}=(\dim\g)(\deg\,
D+3K-g+1)$. The elements of ${\mathcal T}^D$ are distinguished in
the space $\{T|(T)+{\tilde D}\ge 0\}$ by the following relations.
First at every point $\ga$ we have
\begin{equation}\label{E:t-dim}
\begin{aligned}
   T_{-3}&= 2{\dot z_\ga}\nu\a\a^t\s ,\\
   T_{-2}&= {\dot\nu}\a\a^t\s+\nu({\dot\a}\a^t+\a{\dot\a}^t)\s+
      {\dot z_\ga}(\a\b^t+\b\a^t)\s ,\\
   T_{-1}&= ({\dot\a}\b^t+\a{\dot\b}^t+{\dot\b}\a^t+\b{\dot\a}^t)\s .
\end{aligned}
\end{equation}
Since the elements on the left hand side belong to $\g$,
\refE{t-dim} gives $3\dim\g$ relations. Taking account of
\refE{t-tr} and \refE{t-eig} gives $2n+1$ relations, and
\refE{mov-data} give another $2n+1$ relations. Thus we have
$3\dim\g+4n+2$ relations. These relations contain $4n+2$ free
parameters ${\dot z_\ga}$, ${\dot\nu}$, ${\dot\a}$, ${\dot\b}$.
Thus we actually obtained $3\dim\g$ relations at every point
$\ga$, and the number of those points is $K$, hence we have
$3(\dim\g)K$ relations. We see that $\dim\,{\mathcal
T}^D=(\dim\g)(\deg\, D-g+1)$.

But $\L^D$ has the same dimension. We can count this in a quite
similar way or make use of the \refT{almgrad}.  Assume, we are in
the two-point situation, i.e. $D=-m_+P_++(m_-+g)P_-$ where
$m_->m_+$ for simplicity. Then
$\L^D=\g_{m_+}\oplus\ldots\oplus\g_{m_-}$. By \refT{almgrad}
$\dim\L^D=(\dim\g)(m_--m_++1)$ which exactly is equal to
$(\dim\g)(\deg\, D-g+1)$. We conclude that $\dim\,{\mathcal
T}^D=\dim T_L\L^D$, hence these linear spaces coincide.

Next we prove that if $L$, $M$ are as above then $[L,M]$ possesses
the properties \refE{t-expan}---\refE{t-eig}, i.e. belongs to
${\mathcal T}^D$. The proof is straightforward again. For example,
show \refE{t-eig}. Denote the degree zero term  $[L,M]_0$ of the
commutator by $T_0$. Then in the case $\g=\gl(n)$ we find by a
computation
\begin{equation}\label{E:dest}
   T_0\a=\a(\b^tM_1\a-\mu^tL_1\a)+(L_0-\k)M_0\a+L_1\a(\mu^t\a).
\end{equation}
If we replace $\mu^t\a$ with $-\dot z_\ga$, $M_0\a$ with $-\dot\a$
and denote $\b^tM_1-\mu^tL_1\a$ by $\dot\k$, we obtain
\refE{t-eig}. For other types of $\g$ the expression for $T_0\a$
is more complicated and  we use the relations
\refE{triv}---\refE{add2} to identify it with \refE{t-eig}.

In the case $\g=\spn(2n)$ we get
\begin{align*}
  T_0\a& =\a(\nu\a^t\s M_2\a+\b^t\s M_1\a-\mu^t\s L_1\a-\l\a^t\s L_2\a)
   +(L_0-\k)M_0\a+L_1\a(\mu^t\s\a)  \\
  & +\b\a^t\s M_1\a-\mu\a^t\s L_1\a
\end{align*}
instead \refE{dest}. Making use of the relations $\a^t\s M_1\a=0$,
$\a^t\s L_1\a=0$ we obtain the same result.
\end{proof}
\refL{move} directly implies that whenever $([L,M])+D\ge 0$
outside $\ga$'s, and the equations of moving poles are fulfilled
the Lax equation \refE{Lax} is consistent.

\section{Hierarchies of commuting flows}
\label{S:Hierarch}
In the case of $\g=\gl(n)$ I.Krichever \cite{Klax} has shown that
the Lax operator considered as a function on $\L^D$ yields the
hierarchy of commuting flows on $\L^D$. The generalization of that
result on the classic Lie algebras appeared to be not unique.

In this article, we present two methods of constructing
hierarchies of commuting flows given by Lax operators. They differ
by the space of $M$-operators and the number of $\ga$-points. By
the first method we obtain below (following \cite{Klax}) the $A_n$
elliptic Calogero-Moser system. By the second method we obtain the
$C_n$ and $D_n$ elliptic Calogero-Moser systems.

The first method is as follows. For a divisor $D=\sum m_iP_i$
define a divisor ${\widetilde D}=D+\d\sum_{s=1}^{K}\ga_s$ where
\[
K=\begin{cases}
    ng, & \g=\gl(n),\ \sln(n),\ \so(2n),\ \so(2n+1), \\
    (n+1)g, & \g=\spn(2n)
\end{cases}
\]
and
\[
  \d=\begin{cases}
    1,  & \g=\gl(n),\ \sln(n),\ \so(2n),\ \so(2n+1), \\
    2,  & \g=\spn(2n),
  \end{cases}
\]
($K$ and $\d$ depend on $\g^\diamond$ actually).

Let us define $\ND\subset \overline{\g^\diamond}$ as a subspace of
$M$-operators such that $(M)+{\widetilde D}\ge 0$.
\begin{lemma}\label{L:dimND}
 $\dim\ND=(\dim{\g^\diamond})(\deg D+1)$.
\end{lemma}
\begin{proof}
We compute $\dim\ND$ by Riemann-Roch theorem taking account of
additional relations at the points $\ga$. Those are relations
determining $M_{-2}$, $M_{-1}$ and $M_1$ (the latter only in the
case $\g=\spn(2n)$). The number of those relations at every point
$\ga$ is equal to $\d\dim\g^\diamond$ if $\g\ne\spn(2n)$, and is
equal to $\d\dim\g^\diamond+1$ in the last case. We write down
that number in the form $\d\dim\g^\diamond+r_{\spn}$ where
$r_{\spn}=1$ for $\g=\spn(2n)$ and $r_{\spn}=0$ otherwise. We also
have free parameters $\mu$, $\l$ ($\l$ appears only in the case
$\g=\spn(2n)$). Let $r$ be the number of those parameters for a
fixed $\ga$. We can think that at every $\ga$ there are
$\d\dim\g-r+r_\spn$ relations.

Let us write $K$ in the form $K=lg$ where $l$ is equal to $n$ or
to $n+1$ depending on $\g^\diamond$. We have
\begin{equation}
 \begin{aligned}
   \dim N^D&=(\dim\g^\diamond)(\deg
    D+\d lg-g+1)-(\d\dim\g^\diamond -r+r_{\spn})lg \\
    &=(\dim\g^\diamond)(\deg D+1)
    -(\dim\g^\diamond -(r-r_\spn)l)g.
 \end{aligned}
\end{equation}
Next verify that
\begin{equation}\label{E:remain1}
   \dim\g^\diamond =(r-r_\spn)l.
\end{equation}
Indeed, for $\g^\diamond=\gl(n)$ we have $r=l=n$, $r_\spn=0$ hence
$(r-r_\spn)l=n^2$. If $\g^\diamond=\so(2n+1)$ then $r=2n+1$,
$l=n$, $r_\spn=0$ and $(r-r_\spn)l=(2n+1)n$. At last, if
$\g^\diamond=\tsp$ then
$\dim\g^\diamond=n(2n+1)+(2n+1)=(2n+1)(n+1)$ (the sum of the
dimensions of $\spn(2n)$ and the Heisenberg algebra, see page
\pageref{Heis}). We have $r=2n+2$ in this case ($2n+1$ free
parameters come from $\mu$ and $1$ corresponds to $\l$). Hence
$r-r_\spn=2n+1$ and $(r-r_\spn)l=(2n+1)(n+1)$.

In all cases \refE{remain1} is true.
\end{proof}

Following \cite{Klax} let us fix a point $P_0\in\Sigma$ and local
coordinates $w_0$, $w_i$ in the neighborhoods of the points $P_0$,
$P_i$. Our next goal is to define gauge invariant functions $M_a$
that satisfy the assumptions of \refL{move}. Let us define $a$ as
a triple
\begin{equation}\label{E:times}
  a=(P_i,k,m),\ \ k>0,\ m>-m_i,
\end{equation}
where $k,m$ are integers, $k\equiv 1(\rm{mod}\, 2)$ for
$\g=\so(n)$ and $\g=\spn(2n)$.

By \refL{dimND} for generic $L\in\gb$ there is a unique
$\g^\diamond$-valued meromorphic function $M_a$ such that
\begin{itemize}
\item{} (i) $M_a$ is an $M$-operator;
\item{} (ii) outside the points $\ga$ it has pole at the point
$P_i$ only, and
\[
  M_a(q)=w_i^{-m}L^n(q)+O(1),
\]
i.e. singular parts of $M_a$ and $w_i^{-m}L^n$ coincide;
\item{} (iii) $M_a$ is normalized by the condition $M_a(P_0)=0$.\label{exist_flow}
\end{itemize}
\begin{theorem}\label{T:hierarch}
The equations
 \begin{equation}\label{E:hierarch}
   \partial_aL=[L,M_a],\ \partial_a=\partial/\partial t_a
 \end{equation}
 define a hierarchy of commuting flows on an open set of $\L^D$.
\end{theorem}
For $\g=\gl(n)$ the theorem is formulated and proved in
\cite{Klax}. For $\g=\so(2n),\so(2n+1)$ it is proved in
\cite{Sh_lopa} under slightly different assumptions. Here we
formulate and prove the theorem for all classical Lie algebras in
question including $\g=\spn(2n)$.
\begin{proof}
It follows from (ii) that $([L,M_a])+D\ge 0$, hence by \refL{move}
$[L,M_a]\in T_L\L^D$ and the equation $\partial_aL=[L,M_a]$
defines a flow on $\L^D$.

To prove commutativity of such flows it is sufficient to verify
that $M_{ab}=\partial_aM_b-\partial_bM_a+[M_a,M_b]=0$ identically.
By \refL{algM} $M_{ab}$ is an $M$-operator. Below, we prove that
this $M$-operator is regular at the points of the divisor $D$. By
\refL{dimND} the space of such operators has the same dimension as
$\g^\diamond$. Due to (iii) we obtain $M_{ab}=0$.

Let us prove that $M_{ab}$ is regular at the points of the divisor
$D$. We repeat here the corresponding part of the proof of
\cite[Theorem 2.1]{Klax}. First assume that indices $a$,$b$
correspond to the same point $P_i$, i.e. $a=(P_i,n,m)$,
$b=(P_i,n',m')$. Denote $M_a-w^{-m}L^n$ by $M_a^-$ and
$M_b-w^{-m'}L^{n'}$ by $M_b^-$, then by (ii) $M_a^-$ and $M_b^-$
are regular in the neighborhood of $P_i$. We have
\[
 \begin{aligned}
 \partial_aM_b&=w^{-m'}\partial_a L^{n'}+\partial_aM_b^-=
 w^{-m'}[L^{n'},M_a]+\partial_aM_b^-\\
 &=w^{-m'}[L^{n'},M_a^-]+\partial_aM_b^-
 \end{aligned}
\]
and
\[
 \begin{aligned}
 \left[M_a,M_b\right] &=[M_a^-+w^{-m}L^n,M_b^-+w^{-m'}L^{n'}] \\
          &=w^{-m}[L^n,M_b^-]-w^{-m'}[L^{n'},M_a^-]+[M_a^-,M_b^-].
 \end{aligned}
\]
Hence $M_{ab}=\partial_aM_b^--\partial_bM_a^-+[M_a^-,M_b^-]$ at
the point~$P_i$, which is a regular expression at that point. By
definition $M_{ab}$ is regular also at the other points of $D$.

The proof is similar in the case when $a$ and $b$ correspond to
the different points of $D$.
\end{proof}
\section{Symplectic structure}\label{S:Sympl}

Following the lines of \cite{Klax} we introduce here a symplectic
structure on a certain subspace $\P^D\subset\L^D/G$ where
$G=\exp\g$. We call it {\it Krichever-Phong symplectic structure}.

Let $\Psi$ be the matrix formed by the canonically normalized left
eigenvectors of $L$ (we consider a vector $\psi$ to be canically
normalized if $\sum\psi_i=1$). It is defined modulo permutations
of its rows. We consider $L$ and $\Psi$ as matrix-valued functions
on $\L^D$. Let $\d L$ and $\d\Psi$ denote their external
derivatives which are 1-forms on $\L^D$. In the same way we
consider the diagonal matrix $K$ defined by
\[
   \Psi L=K\Psi,
\]
i.e. formed by the eigenvalues of $L$, and the matrix-valued
1-form $\d K$. Let $\Omega$ be a 2-form on $\L^D$ with values in
the space of meromorphic functions on $\Sigma$ defined by the
relation
\[
   \Omega=\tr(\d\Psi\wedge\d L\cdot\Psi^{-1}-\d
   K\wedge\d\Psi\cdot\Psi^{-1}).
\]
$\Omega$ does not depend on the order of the eigenvalues, hence it
is well-defined on $\L$.

Fix a holomorphic differential $dz$ on $\Sigma$ and define a
scalar-valued 2-form $\w$ on $\L^D$ by the relation
\[
    \w=-\frac{1}{2}\left( \sum\limits_{s=1}^K \res_{\gamma_s}\Omega dz+
    \sum\limits_{P_i\in D}\Omega dz\right).
\]
There is another representation for $\Omega$:
\[
 \Omega=2\d\,\tr\left( \d\Psi\cdot\Psi^{-1}K\right)
\]
which implies that $\w$ is apparently closed. First we want prove
that it is nondegenerate when restricted to the space of Tyurin
parameters, i.e $\w$ yields a symplectic form on this space. We
will point out  a canonical form of that restriction.
\begin{lemma}\label{L:gl_red}
The restriction of $\w$ to the space of Tyurin parameters is of
the form
\[
\w_0=\sum_{s=1}^K (a\,\d z_s\wedge\d\k_s+\d\b_s^t\wedge\d\a_s)
\]
where $a=1$ for $\g=\gl(n)$, $a=2$ for $\g=\so(n)$ and
$\g=\spn(2n)$.
\end{lemma}
\begin{proof}
For $\g=\gl(n)$ the corresponding statement is contained in
\cite[Lemma 2.1, Lemma 4.3 ]{Klax}. It is instructive to reproduce
the proof here. Let $g_s$ be a constant nondegenerate matrix such
that $g_s^{-1}\a=e_1$ where $e_1^t=(1,0,\ldots,0)$. Then for
$L'_s=g_s^{-1}L_sg_s$ we have $(L'_s)_{-1}=e_1f^t$ where
$f^t=\b^tg_s$, and $f_1=0$ since $f^te_1=0$. Apparently, only the
entries of the first row of the matrix $L'_{s,-1}$ are nonzero,
but $(L'_{s,-1})_{11}=0$.

The $e_1$ is an eigenvector for $L_{s0}$ with the eigenvalue $\k$.
For that reason $(L_{s0})^{11}=\k$, $(L_{s0})^{i1}=0$ ($i>1$).

It follows from those remarks that the conjugation by the matrix
\linebreak $f_s={\rm diag}(z-z_s,1,\ldots,1)$ takes
$(z-z_s)^{-1}L_{-1}$ and $L_0$ to holomorphic matrix-valued
functions. Hence the same is true for $L_s'$, and the lemma is
proven with $\Phi_s=f_sg_s$. Let us note for the future that $f_s$
is diagonal.

Let us watch now for the transformations of $\Omega$ corresponding
to the just performed transformations of $L$.

Under the gauge transformation $L'=g^{-1}Lg$, $\Psi'=\Psi g$ the
form $\Omega$ transforms to $\Omega'$ where
$\Omega'=\Omega-2\tr(\d L\wedge\d gg^{-1}+L\d gg^{-1}\wedge\d
gg^{-1})$.

After the first of the above transformation (by the matrix $g_s$)
we obtain
\[
\res_{\ga_s}\Omega' dz=\res_{\ga_s}\Omega dz-2\res_{\ga_s}\tr(\d
L\wedge\d g_sg_s^{-1}+L\d g_sg_s^{-1}\wedge\d g_sg_s^{-1})
\]
Since $g_s$ is constant we have
\[
\res_{\ga_s}\Omega' dz=\res_{\ga_s}\Omega
dz-2\res_{\ga_s}\tr((\d\a_s\cdot\b_s^t+\a_s\cdot\d\b_s^t)\wedge\d
g_sg_s^{-1}+\a_s\b_s^t\d g_sg_s^{-1}\wedge\d g_sg_s^{-1}).
\]
By differentiating the relation $g_s^{-1}\a=e_1$ we obtain
$\d\a=\d g_sg_s^{-1}\a$ . Substituting that to the previous
relation we obtain
\[
\res_{\ga_s}\Omega' dz=\res_{\ga_s}\Omega
dz+2\tr(\d\a_s\wedge\d\b_s^t).
\]
The matrix $L'$ becomes holomorphic under the transformation
${\widehat L}_s=fL'f^{-1}$, hence
\[
0=\res_{\ga_s}\Omega' dz+2\res_{\ga_s}\tr(\d L\wedge f_s^{-1}\d
f_s+L\d f_s^{-1}f_s\wedge\d f_s^{-1}f_s).
\]
As $f$ is diagonal the last term vanishes. Making use of the above
obtained special form of $L'$ and $f_s$ we conclude that
\[
\res_{\ga_s}\Omega' dz=2\d z_s\wedge\d\k_s.
\]
Thus the contribution of the point $\ga_s$ to $\Omega$ is equal to
\[
\res_{\ga_s}\Omega=-2\tr(\d\a_s\wedge\d\b_s^t)-2\d
z_s\wedge\d\k_s,
\]
and the corresponding contribution to $\w$ is equal to
\[
\w_s=-{1\over 2}\res_{\ga_s}\Omega=\d\b_s^t\wedge\d\a_s+\d
z_s\wedge\d\k_s.
\]

The proof for $\so(n)$ and $\spn(2n)$ is more complicated but
similar. In the symplectic case we choose
\[
   \s=\begin{pmatrix}
    \s' & 0  \\
    0   & \s''
  \end{pmatrix}\ \text{where}\
  \s'=\begin{pmatrix}
    0 & 1  \\
    -1   & 0
  \end{pmatrix}
\]
as a matrix giving the symplectic form.

We can send $\a$ to any vector by a nondegenerate matrix $g$ so
that $\s$ is invariant, i.e. $g$ is symplectic. Let us send $\a$
to $e_1$ where $e_1^t=(1,0,\ldots,0)^t$. Then
\[
 \a\a^t\s=
  \begin{pmatrix}
    0     & 1          &    0      &\ldots & 0     \\
    0     & 0          &    0      &\ldots & 0     \\
    0     & 0          &    0      &\ldots & 0     \\
   \vdots & \vdots     & \vdots    &\ddots &\vdots \\
    0     &   0        &    0      &\ldots & 0     \\
  \end{pmatrix}, \ \ \
  (\a\b^t+\b\a^t)\s=
  \begin{pmatrix}
    -\b_2 & 2\b_1      &    *      &\ldots & *     \\
    0     & \b_2       &    0      &\ldots & 0     \\
    0     & \b_3       &    0      &\ldots & 0     \\
   \vdots & \vdots     & \vdots    &\ddots &\vdots \\
    0     & \b_{2n}    &    0      &\ldots & 0     \\
  \end{pmatrix}
\]
where stars denote constants linearly depending on $\b$. Observe
that by $\b^t\s\a=0$ we have $\b_2=0$.

It is easy to check that conjugation by
$f=\begin{pmatrix}
    Z & 0  \\
    0   & E
\end{pmatrix}$, where $Z=\begin{pmatrix}
    z & 0  \\
    0   & z^{-1}
\end{pmatrix}$, multiplies the only nonzero entry of $\a\a^t\s$ by
$z^2$, and the matrix $(\a\b^t+\b\a^t)\s$ by $z$ (except for the
entry $2\b_1$ multiplied by $z^2$). Hence
$f\cdot(\frac{\a\a^t\s}{z^2}+\frac{(\a\b^t+\b\a^t)\s}{z})\cdot
f^{-1}$ is holomorphic.

The matrix $L_0$ is similar to $(\a\b^t+\b\a^t)\s$ by structure.
Its first column is equal to $(\k,0,\ldots,0)^t$ by $L_0e_1=\k
e_1$. Its second row is equal to $(0,-\k,\ldots,0)^t$ by
$(e_1^t\s) L_0=-\k(e_1^t\s)$ (which follows from $L_0^t\s+\s
L_0=0$). Nonzero entries of $L_0$ are multiplied by positive
degrees of $z$ under the conjugation by $f$, except for the right
lower $(2n-2)\times(2n-2)$ corner block which is invariant. Hence
$fL_0f^{-1}$ is holomorphic.

Thus, a singularity could only come from $f\cdot(zL_1)f^{-1}$
because $(L_1)_{21}$ is multiplied by $z^{-2}$. But $\a^t\s
L_1\a=0$ implies $(L_1)_{21}=0$. All other entries of $f\cdot L_1
f^{-1}$ are of order $z^{-1}$ and more, hence multiplied by $z$
become holomorphic.

The matrix $\d f\cdot f^{-1}$ has null entries except for the left
upper $2\times 2$ corner block which is equal to $\begin{pmatrix}
    \d z\cdot z^{-1} & 0  \\
    0   & -\d z\cdot z^{-1}
\end{pmatrix}$. As it follows from what was said of the structure
of $L_0$ the corresponding block of $\d L_0$ is equal to
$\begin{pmatrix}
    \d\k & *  \\
    0   & -\d\k
\end{pmatrix}$. Thus, at a point $\ga_s$, we have the
contribution $\d\b_s^t\wedge\d\a_s$ from the gauge transformation
corresponding to $g$, and $2\d z_s\wedge\d\ga_s$ from the gauge
transformation corresponding to $f$. The total contribution of a
point $\ga_s$ to $\w$ is equal to $2\d
z_s\wedge\d\k_s+\d\b_s^t\wedge\d\a_s$.

In the orthogonal case we must additionally satisfy the relation
$e_1^t\s e_1=0$. Consider $\g=\so(2n)$ for example. Let us choose
$\s$ as in the symplectic case but take $\s'=\begin{pmatrix}
    0 & 1  \\
    1   & 0
  \end{pmatrix}$
and assume $\s''$ to be positive definite. Then $e_1$ (where
$e_1^t=(1,0,\ldots,0)^t$) satisfies the required relation and we
proceed as in the symplectic case, with the same matrix function
$f$. The structure of $L_{-1}=(\a\b^t-\b\a^t)\s$ is similar to
that in the symplectic case, with the almost only difference that
the left upper corner $2\times 2$ block is of the form
$\begin{pmatrix}
    \b_2 & *  \\
    0   & -\b_2
\end{pmatrix}$. This difference is technical and does not affect
the result.
\end{proof}

Let us consider now the contribution of the points
$P_1,\ldots,P_N$. Let $\w_m=-{1\over 2}\res_{_{P_m}}\Omega dz$.
Define $\P^D_0$ as a subspace in $\L^D$ where the 1-form $\d\k\, d
z$ is holomorphic. It is the same as the set of common zeroes of
the following functions on $\L^D$:
\[
T_{ijl}=\res_{P_{i}^l}\left( (z-z(P_i))^jk\, dz\right),\
j=0,\ldots,(m_i-d_i),
\]
where $l$ enumerates sheets of the spectral curve (as a branch
cover of $\Sigma$), $d_i=\ord_{P_i}dz$. Let us notice that the
foliation given by the common level sets of those functions is
invariant with respect to the flows of our commuting hierarchy
since those preserve the spectrum $k$ (see also the next Section).
\begin{lemma}\label{L:wm1}
On the space $\P^D_0$
\begin{equation}\label{E:vklad_m}
\w_m=\res_{_{P_m}}\tr(L\,\Psi^{-1}\d\Psi\wedge\Psi^{-1}\d\Psi)dz.
\end{equation}
\end{lemma}
\begin{proof}
By differentiating the relation $\Psi L=K\Psi$ exclude $\d L$ in
the definition of $\Omega$. We will obtain
\[
\w_m=\res_{_{P_m}}\tr(L\,\Psi^{-1}\d\Psi\wedge\Psi^{-1}\d\Psi-\d
K\wedge\d\Psi\cdot\Psi^{-1}+{1\over
2}K\d\Psi\cdot\Psi^{-1}\wedge\d\Psi\cdot\Psi^{-1})dz.
\]
That $K$ is diagonal implies
$\tr(K\d\Psi\cdot\Psi^{-1}\wedge\d\Psi\cdot\Psi^{-1})=0$, thus the
third summand  in the expression for $\w_m$ vanishes. The second
summand is holomorphic on the space $\P_0^D$, because such is $\d
Kdz$, and the form $\d\Psi\cdot\Psi^{-1}$ is holomorphic always
since $\d\Psi$ and $\Psi$ have the same order at $P_m$ (the point
$P_m$ is immovable and there is no variation in the local
coordinate on $\Sigma$ while taking $\d\Psi$).
\end{proof}

Let $\xi$ be a tangent vector to $\P_0^D$ at a point $L$. It is a
$\g$-valued meromorphic function on $\Sigma$, and
$\d\Psi(\xi)\cdot\Psi^{-1}=\xi$. Hence for any pair $\xi$, $\eta$
of tangent vectors we have
$\w_m(\xi,\eta)=\res_{_{P_m}}\tr(L[\xi,\eta])$. This is a 2-form
of Kirillov type. In general it is degenerate and has symplectic
leaves. In particular in the case when $Ldz$ has a simple pole at
$P_m$ with the residue $L_m$ the $\w_m$ descends to the canonical
Kirillov form on the orbit $O_m$ of the element $L_m\in\g$. We
omit details and refer to the full analogy with \cite{Klax} with
this respect.

Since $\w$ is $G$-invariant it is actually defined on
$\P^D=\P^D_0/G$. Together with what has been already proven in
this section we have the following statement quite similar to the
\cite[Theorem 4.1]{Klax}.
\begin{theorem}\label{T:KPh}
The form
\begin{equation}\label{E:sympl} \w=\sum_s (a\,\d\k_s\wedge\d
z_s+\d\a_s^t\wedge\d\b_s)+\sum_m\w_m .
\end{equation}
is nondegenerate on $\P^D$. Thus it gives a symplectic structure
on $\P^D$.
\end{theorem}

\section{Hamiltonian theory}\label{S:Hamilton}

Following the lines of \cite{Klax} we show here that hierarchies
in the \refT{hierarch} are Hamiltonian for all classical Lie
algebras in question, and compute the corresponding Hamiltonians.

For a vector field $e$ on $\L^D$ let $i_e\w$ be the 1-form defined
by $i_e\w(X)=\w(e,X)$ (where $X$ is an arbitrary vector field). By
definition, a vector field $\partial_t$ is Hamiltonian if
$i_{\partial_t}\w=\d H$ where $H$ is some function called the
Hamiltonian of $\partial_t$.
\begin{theorem}\label{T:Hamil}
Let $\partial_a$ be a vector field defined by \refE{hierarch}.
Then
\[
   i_{\partial_a}\w=\d H_a
\]
where
\[
  H_a=-\frac{1}{k+1}\res_{P_i}\tr(w^{-m}L^{k+1})dz, \ a=(P_i,k,m).
\]
\end{theorem}
Before we start proving the theorem let us consider some
properties of $L$, $\partial_a+M_a$ and their spectra. First of
all these operators commute due to the Lax equation, hence the
rows of $\Psi$ are eigenvectors for both of them.  The diagonal
forms of those two operators are defined as follows:
\begin{equation}\label{E:eigenPsi}
  K=\Psi L\Psi^{-1},\ F_a=\Psi (\partial_a+M_a)\Psi^{-1}.
\end{equation}
Equivalently
\begin{equation}\label{E:eigenPsi1}
  \Psi L=K\Psi ,\ \partial_a\Psi= \Psi M_a-F_a\Psi.
\end{equation}
These relations imply
\begin{equation}\label{E:constK}
  \partial_aK=0.
\end{equation}
As shown above $L$ is conjugated to a function holomorphic at the
$\ga$-points, hence its spectrum $K$ is also holomorphic  there.
But $L$ itself is singular, hence for $L$ in a generic position
$\det\Psi$ has a (simple) zero at every $\ga$-point. Assume
\[
\Psi(z-z_\ga)=\Psi_0+\Psi_1(z-z_\ga)+\ldots,\ \
\Psi^{-1}(z-z_\ga)=\frac{\tilde\Psi_{-1}}{(z-z_\ga)}+\tilde\Psi_0+\ldots
.
\]
Then in a generic position the holomorphy of $\Psi L$ and the
relations $\Psi\Psi^{-1}=id$ are equivalent to the following
relations
\begin{equation}\label{E:gpos}
 \Psi_0\a=0, \ \varepsilon\a^t\s{\tilde\Psi}_0=0, \
 {\tilde\Psi}_{-1}=\a{\tilde\b}^t.
\end{equation}
Observe that if $\varepsilon=0$ (i.e. $\g=\gl(n)$) then
$\nu=\l=0$. Hence
$\nu\a^t\s{\tilde\Psi}_0=\l\a^t\s{\tilde\Psi}_0=0$ as well. For
the same reason $\varepsilon\a^t\s\a=\nu\a^t\s\a=\l\a^t\s\a=0$.
\begin{remark}
In the presence of antiinvolution $\Psi$ is {\it always}
holomorphic at the $\ga$-points. Indeed, if $\Psi$ has a pole at a
point $\ga$, let us operate by the antiinvolution on the above
relation. We obtain $(\s(\Psi^t)^{-1}\s^{-1})(\s
L^t\s^{-1})(\s\Psi^t\s^{-1})=\s K^t\s^{-1}$. By replacing $\s
L^t\s^{-1}$ with $-L$ we obtain $(\s\Psi^t)^{-1} L(\s\Psi^t)=
-K^t$. In the last relation, $-K^t$ still is diagonal, and the
matrix to the left of $L$ is holomorphic.
\end{remark}
\begin{remark} If $L$ has a second order pole at a $\ga$, the
relations $\Psi_0\a=0$ and $\a^t\s\tilde\Psi_0=0$ hold always at
the $\ga$. Indeed, by $\Psi L=K\Psi$ we have
\[
\left(\Psi_0+\Psi_1(z-z_\ga)+\ldots\right)\left(
\nu\frac{\a\a^t\s}{(z-z_\ga)^2}+\ldots\right)=O(1).
\]
Hence $\Psi_0\a\a^t\s=0$. Since $\a\ne 0$ and $\s$ is a
nondegenerate quadratic form, we have $\Psi_0\a=0$.

Consider now the same relation in the form
$L\Psi^{-1}=K\Psi^{-1}$. For the left hand side we obtain by
computation
$L\Psi^{-1}=\nu\frac{\a\a^t\s\tilde\Psi_0}{(z-z_\ga)^2}+O((z-z_\ga)^{-1})$
while the right hand side has the pole of order $-1$, at most. As
above, we conclude that $\a^t\s\tilde\Psi_0=0$.

\end{remark}
\begin{lemma}\label{L:holK}
The matrix-valued functions $K$ and $F_a$ are holomorphic at all
$\ga$-points provided \refE{gpos} hold there.
\end{lemma}
\begin{proof}
\begin{align*}
\Psi L\Psi^{-1}= &
\frac{\nu\Psi_0\a\a^t\s{\tilde\Psi}_{-1}}{(z-z_\ga)^3}
+\frac{\nu\Psi_0\a\a^t\s{\tilde\Psi}_0
+\Psi_0(\a\b^t+\varepsilon\b\a^t)\s{\tilde\Psi}_{-1}
+\nu\Psi_1\a\a^t\s{\tilde\Psi}_{-1}}{(z-z_\ga)^2}\\
+&\frac{\nu\Psi_0\a\a^t\s{\tilde\Psi}_1
+\Psi_0(\a\b^t+\varepsilon\b\a^t)\s{\tilde\Psi}_0
+\Psi_0L_0{\tilde\Psi}_{-1}
+\nu\Psi_1\a\a^t\s{\tilde\Psi}_0}{z-z_\ga} \\
+&\frac{\Psi_1(\a\b^t+\varepsilon\b\a^t)\s{\tilde\Psi}_{-1}
+\nu\Psi_2\a\a^t\s{\tilde\Psi}_{-1}}{z-z_\ga} +O(1).
\end{align*}
The singular part of that expression obviously vanishes under the
conditions \refE{gpos}.

Let us prove now holomorphy of the spectrum of $\partial_a+M_a$.
Indeed, this spectrum is expressed by the matrix
$F_a=\Psi(\partial_a+M_a)\Psi^{-1}$. Using the expansions for
$M_a$ and $\Psi^{-1}$ at $\ga$ we obtain
\begin{align*}
(\partial_a+M_a)&\Psi^{-1}=\frac{(\l\a\a^t\s)(\a{\tilde\b}^t)}{(z-z_\ga)^3}
+\frac{(\partial_az_\ga)\a{\tilde\b}^t+\l\a\a^t\s{\tilde\Psi}_0
  +(\a\mu_a^t+\varepsilon\mu_a\a^t)\s\a{\tilde\b}^t}{(z-z_\ga)^2}\\
&+\frac{(\partial_a\a){\tilde\b}^t+\a(\partial_a{\tilde\b}^t)+
\l\a\a^t\s{\tilde\Psi}_1+(\a\mu_a^t+\varepsilon\mu_a\a^t)\s{\tilde\Psi}_0
+ M_{0a}\a{\tilde\b}^t}{z-z_\ga}+O(1)
\end{align*}
After omitting the terms vanishing by the relations \refE{gpos}
(see also the relations on $\a^t\s\a$ there) we obtain
\[
(\partial_a+M_a)\Psi^{-1}=\frac{(\partial_az_\ga+\mu_a^t\s\a)\a{\tilde\b}^t}
{(z-z_\ga)^2} +
\frac{(\partial_a\a+M_{a0}\a){\tilde\b}^t}{z-z_\ga}+\frac{\a{\hat\b}^t}{z-z_\ga}+O(1)
\]
where $\hat\b^t=\partial_a{\tilde\b}^t+
\l\a^t\s{\tilde\Psi}_1+\mu_a^t\s\Psi_0$. Due to \refE{mov-data}
the first summand vanishes, and the second summand is equal to
$\k\a{\tilde\b}^t$. Hence
$(\partial_a+M_a)\Psi^{-1}=\frac{\a{\hat\b}^t+\k\a{\tilde\b}^t}{z-z_\ga}+O(1)$,
and $\Psi(\partial_a+M_a)\Psi^{-1}=O(1)$ by $\Psi_0\a=0$.
\end{proof}
\begin{proof}[Proof of the \refT{Hamil}] Modulo \refL{holK} the
proof of the theorem is the same as in \cite{Klax}. We give it
here for completeness.

By definition
\[
i_{\partial_a}\w=\w(\partial_a,\cdot)=-{1\over 2}\left(
\sum_{s=1}^K\res_{\ga_s}\Lambda+\sum_{i=1}^N\res_{P_i}\Lambda\right),
\]
where $\Lambda=\Omega(\partial_a,\cdot)$. By
$\d\Psi(\partial_a)=\partial_a\Psi$ and $\d
L(\partial_a)=\partial_aL$ (the evaluation of a differential on a
vector field is the derivative along that vector field) we have
\[
\Lambda=\tr\left(\partial_a\Psi\cdot\d
L\cdot\Psi^{-1}-\d\Psi\cdot\partial_a
L\cdot\Psi^{-1}-\partial_aK\cdot\d\Psi\cdot\Psi^{-1}+\d
K\cdot\partial_a\Psi\cdot\Psi^{-1}\right).
\]
By \refE{eigenPsi1} and the Lax equation
\begin{align*}
\Lambda=&\tr\left((\Psi M_a-F_a\Psi)\d
L\cdot\Psi^{-1}-\d\Psi[L,M_a]\Psi^{-1}+\d K(\Psi
M_a-F_a\Psi)\Psi^{-1}\right)\\
=&\tr\left( M_a\d L-F_a\Psi\d
L\cdot\Psi^{-1}-\d\Psi[L,M_a]\Psi^{-1}+\d K\Psi M_a\Psi^{-1}-\d
KF_a\right).
\end{align*}
Let us transform the middle term. By $\Psi L=K\Psi$ we have
$\d\Psi\cdot L=-\Psi\d L+\d K\Psi+K\d\Psi$. Hence
\begin{align*}
\tr\,\d\Psi[L,M_a]\Psi^{-1}= &\,\tr\left( (\d\Psi\cdot
L)M_a\Psi^{-1}-\d\Psi M_aL\Psi^{-1}\right)\\
=&\,\tr\left( (-\Psi\d L+\d K\Psi+K\d\Psi)M_a\Psi^{-1}-\d\Psi
M_aL\Psi^{-1}\right)\\
=&\,\tr\left( -\Psi\d LM_a\Psi^{-1}+\d K\Psi M_a\Psi^{-1}+K\d\Psi
M_a\Psi^{-1}-\d\Psi M_aL\Psi^{-1}\right).
\end{align*}
The last two terms annihilate because
\[
\tr\left( \d\Psi M_aL\Psi^{-1}\right)=\tr\left( \d\Psi
M_a\Psi^{-1}(\Psi L\Psi^{-1})\right)=\tr\left( \d\Psi
M_a\Psi^{-1}K\right),
\]
and we obtain
\[
\tr\,\d\Psi[L,M_a]\Psi^{-1}=\tr\left( -\d LM_a+\d K\Psi
M_a\Psi^{-1}\right).
\]
Substituting that to the last expression for $\Lambda$ we obtain
\[
\Lambda=\tr\left( 2M_a\d L-F_a\Psi\d L\cdot\Psi^{-1}-\d
KF_a\right).
\]
The last two terms are equal (under the symbol of $\tr$) which can
be obtained by replacing $\Psi\d L$ with $-\d\Psi L+\d
K\Psi+K\d\Psi$. Hence we finally obtain
\[
\Lambda=\tr\left( 2M_a\d L-2\d KF_a\right).
\]
This gives
\begin{equation}\label{E:final}
i_{\partial_a}\w=\sum_{j=1}^N\res_{P_j}\tr(\d K\, F_a)dz-R_a,
\end{equation}
where
\begin{equation}\label{E:R_a}
R_a=\sum_{s=1}^K\res_{\ga_s}\tr(\d LM_a)dz+\sum_{j=1}^N
\res_{P_j}\tr(\d LM_a)dz.
\end{equation}
Observe that the sum over the $\ga$-points in \refE{final}
vanishes because $\d K$ and $F_a$ are holomorphic at the
$\ga$-points (\refL{holK}). Observe also that the $\ga$-points and
the points $P_j$ are not all singularities of the function $F_a$.
Indeed $F_a=-\partial_a\Psi\cdot\Psi^{-1}-\Psi M_a\Psi^{-1}$, and
$\Psi^{-1}$ has poles at the branching points of eigenvalues of
$L$.

On the contrary, $L$ and $M_a$ are holomorphic everywhere except
at the points $\ga_s$ and $P_j$. For this reason all singularities
of the the 1-form $M_a\d L$ are located at those points. Hence
$R_a=0$ as the sum of residues of a meromorphic 1-form over all
its poles. Moreover, by the construction of $M_a$ the matrix $F_a$
is holomorphic at $P_j$, $j\ne i$ for all times $a$ corresponding
to the point $P_i$. For this reason for such times
\[
i_{\partial_a}\w=\res_{P_i}\tr(\d K\, F_a)dz.
\]
But $F_a=-\partial_a\Psi\cdot\Psi^{-1}-\Psi M_a\Psi^{-1}$, and
$\partial_a\Psi\cdot\Psi^{-1}$ is holomorphic at $P_i$ because the
coordinate of $P_i$ is independent of any time, hence
$\partial_a\Psi$ and $\Psi$ have the same order at $P_i$. Hence
$F_a=-\Psi M_a\Psi^{-1}+O(1)$. By definition $M_a=w^{-m}L^k+O(1)$
at $P_i$ for $a=(P_i,k,m)$. That implies $\Psi
M_a\Psi^{-1}=w^{-m}\Psi L^k\Psi^{-1}+O(1)$ since $\Psi$ is
holomorphically invertible  at $P_i$. But $\Psi L^k\Psi^{-1}=K^k$,
hence $F_a=w^{-m}K^k+O(1)$ at $P_i$. Since $\d Kdz$ is holomorphic
at $P_i$ we obtain
\[
\begin{aligned}
i_{\partial_a}\w=-&\res_{P_i}\tr(w_i^{-m}\d K\,
K^k)dz=-\frac{1}{k+1}\res_{P_i}\d\tr(w_i^{-m}K^{k+1})= \\
=&-\frac{1}{k+1}\d\res_{P_i}\tr(w_i^{-m}L^{k+1})= \d H_a.
\end{aligned}
\]
\end{proof}
The Hamiltonians $H_a$ are in involution since they depend only on
the spectral parameters. We refer to \cite{Klax} for the details.

\section{Examples: Calogero-Moser systems}
\label{S:CMoser}


Let us start with the example considered in \cite{Klax} --- the
elliptic Calogero-Moser model for $\g=\gl(n)$. Define the Lax
operator by
\begin{equation}\label{E:matrLgl}
 L_{ij}=f_{ij}\frac{\s(z+q_j-q_i)\s(z-q_j)\s(q_i)}{\s(z)\s(z-q_i)\s(q_j-q_i)\s(q_j)}\ (i\ne j), \ \
 L_{jj}=p_j
\end{equation}
where $f_{ij}\in\C$ are constant. This form of $L$ is determined
by two requirements: that $L$ is elliptic and that it has poles at
the points $z=q_i$ ($i=1,\ldots,n$) and $z=0$. The last is the
only immovable pole. By reduction of the remaining gauge freedom
it is obtained in \cite{Klax} that $f_{ij}f_{ji}=1$.  For the
second order Hamiltonian corresponding to that pole we have
according to \refT{Hamil}, up to normalization,
\[
H=\res_{z=0}\,\,\, z^{-1}\!\left(-{1\over 2}\sum_{j=1}^n
p_j^2-\sum\limits_{i<j} L_{ij}L_{ji}\right).
\]
By the addition theorem for Weierstrass functions
\[
-L_{ij}L_{ji}=
\frac{\s(z+q_i-q_j)\s(z+q_j-q_i)}{\s(z)^2\s(q_i-q_j)^2}=\wp(q_i-q_j)-\wp(z),
\]
hence
\[
H=-{1\over 2}\sum_{j=1}^n p_j^2+\sum\limits_{i<j} \wp(q_i-q_j).
\]


Let us consider now the case  of $\so(2n)$. We present here
another method of constructing the hierarchies on elliptic curves.
Let $M$-operators take values in the same algebra, i.e. an
$M$-operator is of the form
\[
  M=\begin{pmatrix}
    A & B \\
    C & -A^t
  \end{pmatrix},\ B^t=-B,\ C^t=-C.
\]
We assume that $K=2n$, $(A)+D+\sum\limits_{i=1}^{n} q_i\ge 0$,
$(C)+D+\sum\limits_{i=1}^{n} q_i\ge 0$ and
$(B)+D+\sum\limits_{i=1}^n (-q_i)\ge 0$ where $D=\sum m_iP_i$ as
earlier. Hence the submatrices $A$, $C$ are holomorphic at the
points $-q_i$, and the submatrix $B$ is holomorphic at the points
$q_i$. Let us denote the space of such $M$-operators by ${\mathcal
N}^D$ again.
\begin{lemma}\label{L:dimND_so}
$\dim\ND=(\dim\g)(\deg D+1)$.
\end{lemma}
\begin{proof}
Denote the dimension of the subspace of elements in $\g$ with
$B=0$ by $d'$, and of the subspace of elements in $\g$ with
$A=C=0$ by $d''$. Thus $\dim\g=d'+d''$. By Riemann-Roch theorem
for $g=1$ we obtain
\begin{equation}\label{E:dimND_so} \dim{\mathcal
N}^D=d'(\deg D+n)+d''(\deg D+n)-(d'-n)n-(d''-n)n-n.
\end{equation}
The last three summands correspond to the relations. At a point
$q_i$ we have $d'$ relations $\res_{q_i}\begin{pmatrix}
    A & 0 \\
    C & -A^t
  \end{pmatrix}=\begin{pmatrix}
    \a_{i}' \\
    \a_{i}''
  \end{pmatrix}
  \begin{pmatrix}
    \mu_{i}'^t, & \mu_{i}''^t
  \end{pmatrix}\s -
  \begin{pmatrix}
    \mu_{i}' \\
    \mu_{i}''
  \end{pmatrix}
  \begin{pmatrix}
    \a_{i}'^t, &   \a_{i}''^t
  \end{pmatrix}\s$ where $\s=
  \begin{pmatrix}
    0 & E \\
    E & 0
  \end{pmatrix}
$, $E$~is the unit matrix. We also have $2n$ parameters
$\mu_{i}'$, $\mu_{i}''$ subjected to $n$ relations
$\a_{i}'\mu_{i}'^t-\mu_{i}'\a_{i}'^t=0$ which are equivalent to
vanishing the $B$-block.

At a point $-q_{i}$ we have $d''$ relations
$\res_{-q_{i}}B=(\a_{-i}\mu_{-i}^t-\mu_{-i}\a_{-i}^t)\s$ (where
$\a_{-i}=
  \begin{pmatrix}
    \a_{-i}' \\
    0
  \end{pmatrix}
$, $\mu_{-i}=\begin{pmatrix}
     \mu_{-i}' \\
    0
  \end{pmatrix}$, and $n$ free parameters coming from
$\mu_{-i}'$.

The last $n$ relations arise due to the fact that the motions of
$q_i$ and $-q_i$ are related. Since $\dot q_i=
-\mu_{i}'^t\a_{i}''-\mu_{i}''^t\a_{i}'$, $-\dot q_i=
-\mu_{-i}'^t\a_{-i}'$, and $\dot q_i+(-\dot q_i)=0$ we obtain
those additional $n$ relations.

Finally
\[
\dim{\mathcal N}^D=(\dim\g)\deg D+(2n^2-n)=(\dim\g)(\deg D+1).
\]
\end{proof}
We take $L$ in the same form as $M$ where $A$ is given by
\refE{matrLgl}. For $i<j$ we take
\begin{equation}\label{E:matrBC}
B_{ij}=f^B_{ij}\frac{\s(z+q_j+q_i)\s(z-q_j)}{\s(z)\s(z+q_i)\s(q_i+q_j)},
\ \
C_{ji}=f^C_{ji}\frac{\s(z-q_j-q_i)\s(z+q_i)}{\s(z)\s(z-q_j)\s(q_i+q_j)}
\end{equation}
where $f^B_{ij},f^C_{ij}\in\C$ are constant. These relations
determine the matrices $B$ and $C$ due to skew-symmetry. Similar
to the case of $\gl(n)$ we obtain $f^B_{ij}f^C_{ji}=-1$ by
reduction of the remaining gauge freedom (taking account of the
relation $\a^t\s\a=0$ which descends to $\a'^t\a''=0$ in this
case). For the Hamiltonian we have
\begin{align*}
H=&-\res_{z=0}\,\,\, z^{-1}\left(\sum_{i=1}^n p_i^2+2\sum_{i<j}
A_{ij}A_{ji}+2\sum_{i<j}B_{ij}C_{ji}\right)\\
=&-\sum_{i=1}^n
p_i^2+2\sum_{i<j}\wp(q_i-q_j)+2\sum_{i<j}\wp(q_i+q_j).
\end{align*}

Let us consider now the case $\g=\spn(2n)$. Define $\ND$ as the
space of $M$-operators taking values in $\g^\diamond=\tsp$, with
$n+1$ pairs of double poles $\pm q_i$ ($i=1,\ldots,n+1$).  Thus $
M=\left(\begin{smallmatrix}
    0 & -a^t\quad\!\!\! b^t & c \\
    0 & \boxed{\begin{smallmatrix}
                A & \ B  \\
                C & -A^t
  \end{smallmatrix}} & \begin{smallmatrix} b_{_{_{_{}}}}\\ a_{_{}} \end{smallmatrix}   \\
    0 & 0        & 0
  \end{smallmatrix}\right)$
where $a,b\in\C^{n}$, $c\in\C$, $A$, $B$, $C$ are $n\times n$
matrices, $B=B^t$, $C=C^t$. We assume that
$(A)+D+\sum\limits_{i=1}^{n+1} 2q_i\ge 0$ (and the same for the
divisors of $C$, $a$,$b$, $c$) and $(B)+D+\sum\limits_{i=1}^{n+1}
2(-q_i)\ge 0$.

The counterpart of the relation \refE{dimND_so} writes
\begin{align*}
  \dim{\mathcal N}^D =& d'(\deg
D+2n+2)+d''(\deg D+2n+2)  \\
  & -(2d'-n-1)(n+1)-(2d''+1-n)(n+1)-n.
\end{align*}
where $d''$ corresponds to the (matrix) dimension of the $B$-block
and $d'$ corresponds to the remainder of the matrix. To explain
the second line let us notice that at each pole $q_1,\ldots,q_n$
we have $2d'$ relations corresponding to the fact that the form of
$-1$ and $-2$ order terms is prescribed. We want to have $\nu=0$
which leads to one more relation following from \refE{mov-n} but
this relation is compensated as follows: we actually don't care of
asymptotical behavior of the entry $c$ in the matrix $M$ because
it corresponds to the center of $\tsp$, so we can omit the
relation on the second order pole for this entry. We also have
$2n+1$ parameters, given by $\mu$, subjected to $n$ vanishing
conditions for the $\res_{q_i}B$. Thus we have $2d'-(n+1)$
effective relations at every $q_i$. Let us notice also that the
parameters $\l$ are compensated by the relations $\a^t\s M_1\a=0$.
At a point $-q_i$ we have $2d''+1-n$ relations since $\mu$ is
$n$-dimensional for the block $B$, and there is nothing to
compensate the relation following from \refE{mov-n}. The last
$(-n)$ corresponds to the relations $\dot q_i+(-\dot q_i)=0$,
$i=1,\ldots,n$ (and we do not care of the behavior of the last
pair of poles).

Thus we obtain
\[
\dim{\mathcal N}^D= (\dim\g^\diamond)(\deg D)+2n^2+n.
\]
Let us notice that $2n^2+n=\dim\g$, i.e. this is the dimension of
the submatrix $\left(\begin{smallmatrix}
                A & \ B  \\
                C & -A^t
  \end{smallmatrix}\right)$
of $M$. Thus we can require that this submatrix vanished at $P_0$
instead of the normalization condition (iii), page
\pageref{exist_flow}, in the construction of the flows $M_a$.

Let us take $L$ in the same form as $M$ where the corresponding
elements $A_{ij}$, $B_{ij}$, $C_{ij}$ are defined by
\refE{matrLgl} (with $A$ instead $L$),\refE{matrBC}. The relations
\refE{matrBC} make sense for $i=j$ and their contribution to the
second order Hamiltonian is equal to
\[
B_{ii}C_{ii}=f_{ii}^Bf_{ii}^C(\wp(2q_i)-\wp(z))
\]
where we can set $f_{ii}^Bf_{ii}^C$ to a constant $\varepsilon$.
The submatrices $a$ and $b$ of the matrix $L$ do not contribute
into the Hamiltonians regardless to any explicit form of them.
Hence
\[
H = -\sum_{i=1}^n
p_i^2+2\sum_{i<j}\wp(q_i-q_j)+2\sum_{i<j}\wp(q_i+q_j)+\varepsilon\sum\limits_{i=1}^n
\wp(2q_i).
\]
which is a conventional form of the second order Hamiltonian of
the elliptic Calogero-Moser model in the symplectic case.



\begin{thebibliography}{10}
\define\PL{Phys. Lett. B}
\define\NP{Nucl. Phys. B}
\define\LMP{Lett. Math. Phys. }
\define\JGP{JGP}
\redefine\CMP{Commun. Math. Phys. }
\define\JMP{J.  Math. Phys. }
\define\Izv{Math. USSR Izv. }
\define\FA{Funktional Anal. i. Prilozhen.}
\def\Pnas{Proc. Natl. Acad. Sci. USA}
\def\PAMS{Proc. Amer. Math. Soc.}
\bibitem{dHPh}
D'Hocker, E., Phong, D.H. \emph{Calogero-Moser Lax pairs with
spectral parameter for general Lie algebras}. Hep-th/9804124.


\bibitem{Klax}
Krichever, I.M. \emph{Vector bundles and Lax equations on
algebraic curves}. Comm. Math. Phys. 229, 229--269 (2002).


\bibitem{Kr_ellCM}
Krichever, I.M. \emph{Elliptic solutions to the
{K}adomtsev-{P}etviashvili equation and integrable systems of
particles}. Fuct. Analysis and Appl. 4,\textbf{14} (1980), p.
45--54.


\bibitem{rKNU}
Krichever I.M., Novikov S.P. \emph{Holomorphic bundles on
algebraic curves and nonlinear equations}. Uspekhi Math. Nauk
(Russ. Math.Surv), {\bf 35} (1980),  6, 47--68.

\bibitem{KNfaa}
Krichever I.M., Novikov S.P. \emph{Holomorphic bundles on Riemann
surfaces and Kadomtsev-Petviashvili equation.1}. Funct. Anal. and
Appl., {\bf 12} (1978),  4, 41--52.


\bibitem{KrComm}
Krichever I.M. \emph{Commutative rings of ordinary linear
differential operators}. Functional Anal. and Appl. (Russ.), {\bf
12} (1978), n 3, 20--31.

\bibitem{KNFa}
Krichever, I.M. Novikov, S.P. \emph{Algebras of {V}irasoro type,
{R}iemann
  surfaces and structures of the theory of solitons}. Funktional Anal. i
  Prilozhen. \textbf{21}, No.2 (1987), 46-63.

\bibitem{KSlax}
Krichever, I.M., Sheinman, O.K. \emph{Lax operator algebras}.
Funct. Anal. i Prilozhen., {\bf 41} (2007), no. 4, p. 46-59.
math.RT/0701648.

\bibitem{Perelomov}
Perelomov, A.M. \emph{{I}ntegrable systems of classical mechanics
and {L}ie algebras}, Birkh{\"a}user Verlag, Basel, 1990.


\bibitem{SSlax}
Schlichenmaier, M., Sheinman, O.K. {\it {C}entral extensions of
{L}ax operator algebras}. Russ. Math. Surv., {\bf 63}, no.4, p.
131-172. ArXiv:0711.4688.



\bibitem{ShN65}
Sheinman, O.K. \emph{Krichever-Novikov algebras, their
representations and applications}. In: Geometry, Topology and
Mathematical Physics. S.P.Novikov's Seminar 2002-2003,
V.M.Buchstaber, I.M.Krichever, eds., AMS Translations, Ser.2, v.
212 (2004), 297--316, math.RT/0304020.

\bibitem{ShN70AMS}
Sheinman, O.K. \emph{On certain current algebras related to
finite-zone integration}. In: Geometry, Topology and Mathematical
Physics. S.P.Novikov's Seminar 2004-2008, V.M.Buchstaber,
I.M.Krichever, eds., AMS Translations, Ser.2, v.224 (2008).

\bibitem{Sh_lopa}
Sheinman, O.K. \emph{Lax operator algebras and integrable
hierarchies}. In: Proc. of the Steklov Institute of Mathematics,
2008, v.263.

\bibitem{Tyvb}
Tyurin, A.N. \emph{Classification of vector bundles on an
algebraic curve of an arbitrary genus}. Soviet Izvestia, ser.
Math., \textbf{29}, 657--688.

\end{thebibliography}
\end{document}